\documentclass[11pt]{article}
\usepackage{amsmath}
\usepackage{amssymb}
\usepackage{amsthm}
\usepackage{enumerate}
\usepackage{setspace}
\usepackage{graphicx,color,epsfig}
\usepackage[margin=1in]{geometry}
\parskip \medskipamount
\newtheorem{theorem}{Theorem}[section]
\newtheorem{lemma}[theorem]{Lemma}

\newcommand{\ind}[1]{\mathbf{1}\left(#1\right)}
\newcommand{\pr}[1]{\operatorname{\bbP}\left(#1\right)}

\newcommand{\var}{\operatorname{Var}}
\newcommand{\G}{\Gamma}
\renewcommand{\r}{\rho}

\newcommand{\sobolev}{c_\mathrm{S}}
\newcommand{\compl}{\mathrm{c}}

\newcommand{\rel}{\mathrm{rel}}
\newcommand{\tmix}{T_\mathrm{mix}}
\newcommand{\trel}{T_\mathrm{rel}}
\renewcommand{\O}{\Omega}
\renewcommand{\L}{\Lambda}
\newcommand{\mix}{\mathrm{mix}}
\newcommand{\ent}{\mathrm{Ent}}
\newcommand{\ml}{\mathrm{\ell}}

\newcommand{\mr}{\mathrm{r}}

\newcommand{\Lokn}{\Lambda^0_{n,k}}
\newcommand{\Lkn}{\Lambda_{n,k}}
\newcommand{\Okn}{\Omega(n,k)}
\newcommand{\si}{\sigma}
\renewcommand{\t}{\tau}
\renewcommand{\l}{\lambda}

\newcommand{\IGNORE}[1]{}

\renewcommand{\tilde}{\widetilde}
\renewcommand{\hat}{\widehat}
\newcommand{\cC}{\ensuremath{\mathcal C}} 
\newcommand{\cD}{\ensuremath{\mathcal D}} 
\newcommand{\cE}{\ensuremath{\mathcal E}} 
\newcommand{\cF}{\ensuremath{\mathcal F}} 
\newcommand{\cG}{\ensuremath{\mathcal G}}

\newcommand{\bbN}{{\ensuremath{\mathbb N}} } 
 
\newcommand{\bbP}{{\ensuremath{\mathbb P}} } 
 
\newcommand{\bbR}{{\ensuremath{\mathbb R}} }

\newcommand{\polylog}{\mathrm{polylog}}

\begin{document}

\title{Dynamics of Lattice Triangulations on Thin Rectangles}
\author{Pietro Caputo,
%\thanks{Department of Mathematics, University of Roma Tre, Largo San Murialdo~1, 00146~Roma, Italy.  Email: {\tt caputo@mat.uniroma3.it}}, 
        Fabio Martinelli\thanks{Department of Mathematics, University of Roma Tre, Largo San Murialdo~1, 00146~Roma, Italy.   {\tt caputo@mat.uniroma3.it}, {\tt martin@mat.uniroma3.it}}, 
        Alistair Sinclair\thanks{Computer Science Division, University of California, Berkeley CA~94720-1776, U.S.A.  {\tt sinclair@cs.berkeley.edu}}, 
        Alexandre Stauffer\thanks{Department of Mathematical Sciences, University of Bath, U.K. {\tt a.stauffer@bath.ac.uk}. Supported in part by a Marie Curie Career Integration Grant PCIG13-GA-2013-618588 DSRELIS.}}
%\address{Pietro Caputo: Department of Mathematics, University
%of Roma Tre, Largo San Murialdo~1, 00146~Roma, Italy.  Email: {\tt caputo@mat.uniroma3.it}}
%\author[Fabio Martinelli]{Fabio Martinelli}
%\address{Fabio Martinelli: Department of Mathematics, University
%of Roma Tre, Largo San Murialdo~1, 00146~Roma, Italy.  Email: {\tt martin@mat.uniroma3.it}}
%\email{martin@mat.uniroma3.it}
%\author[Alistair Sinclair]{Alistair Sinclair}
%\address{Alistair Sinclair: Computer Science Division, University
%        of California, Berkeley CA~94720-1776, U.S.A.  Email: {\tt sinclair@cs.berkeley.edu}}
%\author[Alexandre Stauffer]{Alexandre Stauffer}
%\address{Alexandre Stauffer: Department of Mathematics, University
%of Roma Tre, Largo San Murialdo~1, 00146~Roma, Italy.  Email: {\tt alexandrestauffer@gmail.com}}

%\author{}
\date{}
\maketitle
\thispagestyle{empty}

\begin{abstract}
We consider random lattice triangulations of $n\times k$ rectangular regions with weight $\l^{|\si|}$ where $\l>0$ is a parameter and 
$|\si|$ denotes the total edge length of the triangulation. 
When $\l\in(0,1)$ and $k$ is fixed, we prove a tight upper bound of order $n^2$ for the mixing time of the edge-flip Glauber dynamics. Combined with the previously known lower bound of order $\exp(\O(n^2))$ for $\l>1$~\cite{CMSS14}, this establishes the existence of a dynamical phase transition for thin rectangles with critical point at $\l=1$.  
% \newline
% \newline
% \emph{Keywords and phrases.} Poisson point process, Brownian motion, fractal percolation, multi-scale analysis.
% \newline
% MSC 2010 \emph{subject classifications.}
% Primary 82C43; %Time-dependent percolation
% Secondary 60G55, % Point processes
%           60J65, % Brownian motion
%           60K35, % Interacting random processes; statistical mechanics type models; percolation theory
%           82C21. % Dynamic continuum models (systems of particles, etc.)
\end{abstract}

%\section*{Full Paper}
%\setcounter{footnote}{0}
%\setcounter{page}{1}

%############################################################################################
%############################################################################################
%############################################################################################
%\section{Introduction}
%\begin{itemize}
%\item Basics
%\item Problems
%\item Previous results
%\item Details on the model + explain the key ingredients from Lyapunov function
%\item Main results (Th 3.1, 4.3, 5.1)
%\item Few main ideas (log-Sobolev+burn-in)
%\end{itemize}

\section{Introduction}\label{sec:intro}
Consider an $n\times k$ lattice  rectangle $\Lokn = \{0,1,\ldots,n\}\times\{0,1,\ldots,k\}$ in the plane. A triangulation of $\Lokn$ is defined as a maximal set of non-crossing edges
(straight line segments), each of which connects two points of~$\Lokn$ and
passes through no other point. See Figure \ref{fig1} for an example. 

\begin{figure}[h]
\centerline{
\psfig{file=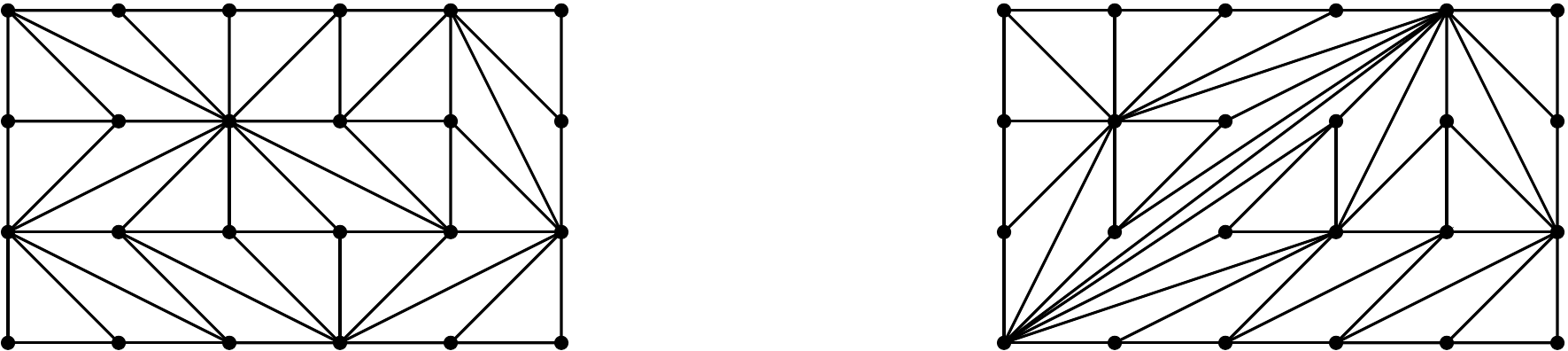,height=1in,width=4.7in}}
\caption{\it Two triangulations of a $5\times 3$ rectangle}
\label{fig1}
\end{figure}

Call $\Okn$ the set of all triangulations of $\Lokn$. All $\si\in\Okn$ have the same number of edges and the set of midpoints of the edges of  $\si$ does not depend on $\si$. Thus, we may view $\si\in\Okn$ as a collection of variables $\{\si_x, \;x\in\Lkn\}$, where 
$$ \Lkn := {\textstyle \{0,\frac{1}{2},1,\frac{3}{2},\ldots,n-\frac{1}{2},n\}
                 \times \{0,\frac{1}{2},1,\frac{3}{2},\ldots,k-\frac{1}{2},k\}}\setminus\Lokn, $$
                 is the set of all midpoints. Moreover,
                 any element $\si\in\Okn$ is {\em unimodular}, i.e., each triangle in $\si$ has area $\tfrac12$; see, e.g., \cite{KZ,tr_book,CMSS14} for these standard structural properties.
If an edge $\si_x$ of $\si$ is the diagonal of a parallelogram, then it is said to be {\em flippable}: one can delete this edge and add the opposite diagonal to obtain a new triangulation $\si'\in\Okn$.  In this case $\si,\si'$ differ by a single {\em diagonal flip} and are said to be {\em adjacent}. The corresponding graph with vertex set $\Okn$, and edges between adjacent triangulations, called the {\em flip graph}, is known to be connected and to have interesting structural properties; see \cite{KZ,CMSS14} and references therein.

We consider the following model of random triangulations. Fix $\l\in(0,\infty)$ and define a probability  measure $\mu $ on $\Okn$ by  
\begin{equation*}
   \mu(\sigma) = \frac{\lambda^{|\sigma|}}{Z},
\end{equation*} 
where $Z=\sum_{\si'\in\Okn} \lambda^{|\sigma'|}$ and $|\sigma|$ is the total $\ell_1$~length of the
edges in~$\sigma$, i.e., the sum of the horizontal and vertical lengths of each edge.  The case $\lambda=1$
is the uniform distribution, while $\lambda<1$ (respectively, $\lambda>1$) favors
triangulations with shorter (respectively, longer) edges. We refer to \cite{CMSS14} and references therein for background and motivation concerning this choice of weights.

A natural way to simulate triangulations distributed according to $\mu$ is to use the edge-flip {\em Glauber dynamics} defined as follows. In state $\sigma$, pick a midpoint $x\in\Lkn$ uniformly at random; if the edge~$\sigma_x$ is flippable to edge $\sigma'_x$ (producing a new
triangulation~$\sigma'$), then flip it with probability \begin{equation}\label{probs}
        \frac{\mu(\sigma')}{\mu(\sigma')+\mu(\sigma)} =
        \frac{\lambda^{|\sigma'_x|}}{\lambda^{|\sigma'_x|} + \lambda^{|\sigma_x|}}, \end{equation}
else do nothing.  Since the flip graph is connected, this defines an irreducible Markov chain on~$\Okn$, and the flip probabilities \eqref{probs} ensure that the chain is reversible with respect to~$\mu$.  Hence the dynamics converges to the stationary
distribution~$\mu$.  We analyze convergence to stationarity
via the standard notion of {\it mixing time}, defined by
 \begin{equation}\label{eq:tmix}
\tmix = \inf\big\{ t\in\bbN \,:\;\max_{\si\in\Okn}\|p^t(\si,\cdot) - \mu\|\leq 1/4
\big\} \,,
\end{equation}
where $p^t(\si,\cdot)$ denotes the distribution after $t$ steps when the initial state is $\si$, and 
$\|\nu-\mu\|=\frac12\sum_{\si\in\Okn}|\nu(\si)-\mu(\si)|$ is the usual total variation distance between
two distributions~$\mu,\nu$. 
%
%Start with a given triangulation $\si$ and pick a
%midpoint $x\in\Lkn$ uniformly at random; if $\si_x$ is flippable then we update its value 
 
%############################################################################################
%############################################################################################
%############################################################################################

As discussed in \cite{CMSS14}, there is empirical evidence that the value $\l=1$ represents a {\em critical point} separating the {\em sub-critical regime} $\l\in(0,1)$, characterized by rapid decay of both equilibrium and dynamical correlations, from the {\em super-critical regime} $\l>1$, characterized by
the emergence of long-range correlations and a dramatic slowdown in the convergence to equilibrium. 
We substantiated this picture by showing that there exist constants $C>0$ and  $\l_1\in(0,1)$ such that $$\tmix \leq Ckn(k+n),$$ for all $k,n\in\bbN$ and for all $\l\leq \l_1$; see  \cite[Theorem 5.1]{CMSS14}. This estimate is based on a coupling argument that requires $\l$ to be sufficiently small; in particular, 
$\l_1 = 1/8$ suffices. We conjectured in~\cite{CMSS14} that the mixing time should satisfy $\tmix = O(kn(k+n))$ throughout the sub-critical regime $\l\in(0,1)$.
However, except for the special case $k=1$, establishing even an arbitrary polynomial bound on $\tmix$ in the whole region $\l<1$ has turned out to be very challenging.  Regarding the super-critical regime, by 
\cite[Theorem 6.1 and Theorem 6.2]{CMSS14} it is known that, for $\l>1$, one has $\tmix = \exp(\O(k+n))$ for all $k,n$, and that $\tmix = \exp(\O(n^2/k))$ if $n>k^2$ . 

In this paper 
% we present new results concerning the sub-critical regime. In particular, 
we establish the conjectured behavior for all $\l<1$ in the case of ``thin" rectangles, i.e., the case when $k$ is fixed and $n$ is large. 
\begin{theorem}\label{mainth}
   For any $\l\in(0,1)$, $k\in\bbN$, there exists a constant $C=C(\lambda,k)>0$ such that the mixing time of the Glauber dynamics for $n\times k$ triangulations satisfies $\tmix \leq C\,n^2$ for all $n\geq 1$. 
\end{theorem}
We remark that the above bound is sharp up to the value of the constant $C$ since it is known that $\tmix\geq C_0 kn(k+n)$ for some positive constant $C_0$ for any $k,n\in\bbN$ and any $\l>0$; see \cite[Proposition 6.3]{CMSS14}. However, as a function of $k$ the constant $C$ in Theorem \ref{mainth} can be exponentially large, and thus the interest of this bound is limited to the case of thin rectangles. 

In the special case $k=1$, the above theorem can be obtained by a direct coupling argument; see \cite[Theorem 5.3]{CMSS14}. Moreover, it is interesting  to observe that in the case $k=1$ the set of triangulations is in 1-1~correspondence with the set of configurations of a lattice path, and that diagonal flips are equivalent to so-called mountain/valley flips in the lattice path representation. Weighted versions of lattice path models have been studied extensively in the past (see, e.g., \cite{CMT,GPR}), and it is tempting to analyze the 
$n\times k$ triangulation model as a multi-path system with $k$ interacting lattice paths. While this can be done in principle, it turns out that the interaction between the paths is technically very complex. Even the case $k=2$ apparently does not allow for significant simplification with  this representation. 

The proof of Theorem~\ref{mainth} will rely crucially on some recent developments by one of us~\cite{AS} based on a Lyapunov function approach to the sub-critical regime $\l\in(0,1)$. As detailed in subsequent sections,  the main results of \cite{AS} will be used first  to show that after $T=O(n^2)$ steps of the chain we can reduce the problem to a restricted chain on a ``good'' set of triangulations, each edge of which never exceeds logarithmic length, and then to show that distant regions in our thin rectangles can be decoupled with an exponentially small error. This will enable us to set up a recursive scheme for functional inequalities related to mixing time such as the {\em logarithmic Sobolev inequality}. %\cite{DiSa}. 
The recursion,  based on a bisection approach for the relative entropy functional inspired by the spin system analysis of \cite{Martinelli:1999jm,Cesi}, allows us to reduce the scale from $n\times k$ down to $\polylog(n)\times k$.
Once we reach the $\polylog(n)\times k$ scale, we use a refinement from~\cite{CLMST} of the classical {\em canonical paths argument}~\cite{AlSi}. This allows one to obtain an upper bound on the relaxation time of a Markov chain in terms of the congestion ratio restricted to a subspace $\O'$ and the time the chain needs to visit~$\O'$ with large probability. 
Here we use a further crucial input from \cite{AS} permitting us to identify a ``canonical" subset of triangulations $\O'$ such that after $T=O(n^2)$ the chain enters $\O'$ with large probability and such that the chain restricted to $\O'$ has small congestion ratio. A detailed high-level overview of the proof will be given in Section~\ref{sec:poverview}.

The rest of the paper is organized as follows. In Section \ref{maintools}, we first recall some important tools from \cite{CMSS14} and then formulate the main ingredients we need from~\cite{AS}. Then, in Section \ref{sec:canonical} we develop the applications of improved canonical path techniques to our setting.  In Section \ref{sec:tightmixing} we discuss the recursive scheme for the log-Sobolev inequality and prove Theorem \ref{mainth}. 
 
\section{Main tools}\label{maintools}

\subsection{Triangulations with boundary conditions} 
We will often consider subsets of $\O(n,k)$ consisting of triangulations in which some edges
are kept fixed, or ``frozen"; we call these {\it constraint edges}.
%When some edges in the triangulation are {\em frozen}, that is they are kept fixed, we speak of {\em constraint edges}. 
Formally, let $\L'\subset\Lkn$ denote a subset of the midpoints, and fix a 
collection of non-crossing edges $\{\t_y,\;y\in\L'\}$, i.e., straight lines with midpoints in $\L'$ each of which connects two points of~$\Lokn$ and
passes through no other point of~$\Lokn$. If $\si\in\Okn$ satisfies $\{\si_y=\t_y\,,\;y\in\L'\}$, we say that $\si$ is {\it compatible\/} with the constraint edges $\t$. We interpret the constraint edges $\t$ as a {\em boundary condition}.

We shall actually need a more general notion of boundary condition, in order to deal with the possibility of constraint edges whose midpoints lie outside the rectangle $\Lokn$. 
Let $N$ be an integer and consider the set $Q^0_{N,n,k}=\{-N,\dots,n+N\}\times \{0,\dots,k\}$, i.e., a $(2N+n)\times k$ rectangle containing $\Lokn$, and let $Q_{N,n,k}$ denote the set of midpoints of a triangulation of $Q^0_{N,n,k}$. 
Fix a triangulation $\hat\t$ of the region $Q^0_{N,n,k}$ and call $\t$ the set of edges obtained from $\hat\t$ by deleting some or all edges $\hat\t_x$ with midpoint $x\in \Lkn$. Thus, $\t$ is a set of constraint edges for triangulations of 
$Q^0_{N,n,k}$ such that all edges with midpoints in $Q_{N,n,k}\setminus \Lkn$ are assigned. 
Given constraint edges 
$\t$ as above, we define $\O^\t(n,k)$ as the set of all triangulations $\si$ of $Q^0_{N,n,k}$ that are
%\marginpar{\tiny AJS: Last occurrence of $Q^0_{N,n,k}$ should be replaced by $\Lkn$ here, right?  $\si$ only contains the edges with midpoint in $\Lkn$.}
compatible with $\t$. Since the parameter $N$ will play no essential role in what follows we often omit it from our notation. Since all elements of $\O^\t(n,k)$ have the same edges at midpoints in $ Q_{N,n,k}\setminus \Lkn$, one can also view a triangulation $\si\in\O^\t(n,k)$ as an assignment of edges to midpoints in $\Lkn$ with certain constraints. 
Note that while the midpoint of a non-constraint edge of a triangulation $\si\in\O^\t(n,k)$ is always contained in $\Lkn$, its endpoints need not be contained in $\Lokn$; we refer to Lemma \ref{lengthle} below for a quantitative statement on the smallest rectangle containing all non-constraint edges of any $ \si\in\O^\t(n,k)$ in terms of the length of the largest edge in $\t$. 

The random triangulation $\si$ with boundary condition $\t$ is the random variable $\si\in\O^\t(n,k)$ with distribution
\begin{equation}\label{eq:gibbs}
   \mu^\t(\sigma) = \frac{\lambda^{|\sigma|}}{Z},
\end{equation} 
where $Z=\sum_{\si'\in\O^\t(n,k)}\lambda^{|\sigma'|}$. 
We sometimes write $\mu$ instead of $\mu^\t$ and $\O$ instead of $\O^\t(n,k)$ if there is no need to stress the dependence on the constraint edges. 
We say that there is {\it no boundary condition\/} when $N=0$ and the set of constraint edges $\t$ is empty. In this case 
$\O^\t(n,k)$ coincides with $\Okn$, the set of all triangulations of $\Lokn$. 
%We regard $\mu$ as a Gibbs distribution, and we interpret the constraint edges $\t$ as a {\em boundary condition}. 
%Fix a triangulation $\t$ of the lattice rectangle $Q_{N,n,k}$  
%For later purposes it will be convenient to consider also triangulations $\si$ of the $n\times k$ rectangle $\Lokn$
%with constraint edges whose midpoints are outside the region $\Lokn$. This can be achieved within the above setting simply by considering triangulations of a larger rectangle $\L^0_{n',k}$, $n'>n$, that contains all midpoints of the constraint edges.   

\subsection{Ground states}
It is a fact that for any set of constraint edges $\t$, the set of triangulations $\O^\t(n,k)$ that are compatible with $\t$ is non-empty. Among the compatible triangulations, we are particularly interested in those with minimal $\ell_1$-edge length, which we call {\em ground state triangulations}. These are the triangulations of maximum weight in \eqref{eq:gibbs} when $\lambda<1$, and they play a central role in our analysis.
In the absence of boundary conditions, the ground state triangulations are
trivial: every edge is either horizontal or vertical or a unit diagonal, so in particular
the ground state is unique up to flipping of the unit diagonals.
The presence of constraint edges can change the ground
state considerably.  However, the following result from \cite[Lemma~3.4]{CMSS14} reveals the 
strikingly simple structure of ground states for any set of contraints.
\begin{lemma}\label{lemma:gsl}{\em [Ground State Lemma]}
Given any set of constraint edges, the ground state triangulation is unique
(up to possible flipping of unit diagonals), and can be constructed by
placing each edge in its minimal length configuration consistent with the
constraints, independent of the other edges.
\end{lemma}

Given a set of constraint edges, we denote by $\bar \si$ the unique ground state triangulation. 
(An arbitrary choice of the available unit diagonal orientations is understood in this notation.) 
If no confusion arises, we omit to specify the dependence on the constraint edges.  
An important structural property of triangulations with constraint edges, which follows
from Lemma~\ref{lemma:gsl}, is that from any triangulation $\si$ compatible with $\t$ one can reach the ground state $\bar\si$ with a path in the flip graph with the property that no flip increases the length
of an edge.

\subsection{The Glauber dynamics}
The Glauber dynamics in the presence of a boundary condition $\t$  is defined as before (see equation~\eqref{probs}), with the modification that the midpoint $x$ to be updated is picked uniformly at random among all midpoints of non-constraint edges. For any $\l>0$, this defines an irreducible Markov chain on $\O^\t(n,k)$ that is reversible w.r.t.\ the stationary distribution $\mu^\t$ (see~\cite{CMSS14} for  details). It was shown in \cite[Theorem 5.1]{CMSS14} that for some constants $C>0$ and $\l_1\in(0,1)$, the mixing time of this chain in an $n\times k$ rectangle satisfies $\tmix \leq Ckn(k+n)$ uniformly in the choice of the constraint edges, whenever $\l\leq \l_1$. 
We also conjectured in~\cite{CMSS14} that the $O(kn(k+n))$ mixing time should hold for all $\l\in(0,1)$. 

\subsection{Key ingredients from \cite{AS}}\label{keyas}
%
%Given two triangulations, $\sigma,\eta$, we denote by $\sigma \sim \eta$ if $\sigma$ and $\eta$ differ by a single flip; 
%in this case we also say that $\sigma$ and $\eta$ are \emph{adjacent} triangulations.
%
We gather in Lemmas \ref{lem:tailedge}--\ref{lem:toptobottom} below some estimates 
from~\cite{AS} that will be crucial in our analysis; for the proofs see~\cite{AS}.  Note that these
estimates are valid throughout the sub-critical regime $\l\in(0,1)$.
%The crucial point of the estimates collected here is that they are valid throughout the whole sub-critical regime $\l\in(0,1)$. For the proofs of Lemmas \ref{lem:tailedge}-\ref{lem:toptobottom} below we refer to \cite{AS}. 

The first lemma applies to the case where there are no constraint edges, so that the ground state is trivial.  
It follows from~\cite[Corollary~7.4]{AS}, and establishes that after running the Markov chain for $O(n^2)$ steps, 
the $\ell_1$-length of a given edge has an exponential tail.  
For a given initial triangulation $\si=\si^0$, we denote by $\si^t$ the triangulation after~$t$ steps of 
the chain. 

\begin{lemma}\label{lem:tailedge}
   Fix $\l\in(0,1)$.  There exist positive
   constants $c_1 = c_1(\lambda)$ and $c_2 = c_2(\lambda)$ such that for $n\geq k\geq 1$, for any $t \geq c_1 n^2$, any $\ell> 0$, any midpoint $x\in\Lkn$, and any initial triangulation $\si\in\O(n,k)$:
   $$
      \pr{|\sigma_x^t| \geq \ell}
      \leq c_1\exp{(-c_2 \ell)}.
   $$
\end{lemma}

The next lemma deals with the evolution in the presence of constraint edges $\t$, and follows from~\cite[Theorem~7.3]{AS}. 
We denote by $\bar\si_x$ the ground state edge at $x$ (compatible with $\t$). 
Given $\si\in\O^\t(n,k)$ and $y\in\Lkn$, we write $\si_y\cap \bar\si_x\neq \emptyset$ if the edge $\si_y$ crosses $\bar\si_x$ 
(not including the case where $\si_y$ and $\bar\si_x$ intersect only at their endpoints). 

\begin{lemma}\label{lem:crosstail}
   Fix $\l\in(0,1)$.  There exist positive
   constants $c_1 = c_1(\lambda)$ and $c_2 = c_2(\lambda)$ such that the following holds for $n\geq k\geq 1$, for any set of constraint edges $\t$.  Let $M$ be the $\ell_1$ length  of the largest edge in any triangulation $\si\in\O^\t(n,k)$. Then, 
for any $t\geq c_1kn(M+\log n)$, and any $\ell\geq 0$, we have 
  \begin{equation}\label{grcr}
      \pr{\bigcup\nolimits_{y\in \Lkn}\big\{\sigma^t_y \cap \bar\sigma_x\neq \emptyset\big\} \cap \big\{|\sigma^t_y| \geq |\bar\sigma_x|+\ell\big\}}
      \leq c_1\exp\left(-c_2\ell\right).
   \end{equation}
\end{lemma}

Next we give a rough upper bound on the number of small edges intersecting a given ground state  edge. 
We assume that a set of constraint edges $\t$ is given. For any triangulation $\sigma\in\O^\t(n,k)$,
%with $\si=\{\si_x\,,\;x\in\Lkn\}$, 
any ground state edge $g$, and any $\ell\in\mathbb{Z}^+$, define
$$
   I_g(\sigma,\ell)=\left\{\sigma_x\,,\;x\in\Lkn\colon \;\sigma_x \cap g \neq \emptyset \text{ and } |\sigma_x|\leq |g|+\ell\right\}.
$$
%the set of edges of $\sigma$ that intersect $g$ and have $\ell_1$ length at most $|g|+\ell$. 
We denote by $|I_g(\sigma,\ell)|$ the cardinality of $I_g(\sigma,\ell)$. For a proof of the lemma below, see~\cite[Proposition~4.4]{AS}.
\begin{lemma}\label{lem:setI}
   Let $g$ be a ground state edge, and let $\sigma\in\O^\t(n,k)$ be a triangulation. 
   \begin{enumerate}[i)]
      \item If $\sigma_x\cap g\neq \emptyset$ then 
         $|\sigma_x|\geq |g|$, with strict inequality when the midpoint of $g$ is not $x$.
      \item For any $\ell\geq 1$, all midpoints of edges in $I_g(\sigma,\ell)$ are contained in the ball of radius $2\ell$ centered at the midpoint of $g$.
      \item There exists a universal $c>0$ such that for any $\ell\geq 1$ we have 
         $$|I_g(\sigma,\ell)| \leq c\,\ell^2\,,\;\;\text{ and} \;\quad\big|\bigcup\nolimits_\sigma I_g(\sigma,\ell)\big| \leq c\,\ell^4.$$
   \end{enumerate}
\end{lemma}

Finally, the lemma below establishes the probability of having a top-to-bottom crossing of unit verticals in a random triangulation $\si$. 
By a ``top-to-bottom crossing of unit verticals in $\si$" we mean 
a straight line of length $k$ made up of $k$ vertical edges in $\si$ each of length $1$. 
The lemma below follows from~\cite[Theorems~8.1 and 8.2]{AS}.
\begin{lemma}\label{lem:toptobottom}
   Let $k\in\bbN$ and $\l\in(0,1)$ be fixed. There exist positive constants $c=c(\l,k)$, $\delta=\delta(\l,k)$ and $m_0=m_0(\l,k)$
    such that the following holds. 
   Let $R$ be an $m \times k$ rectangle inside $\Lokn$ with $m\geq m_0$. 
   Consider an arbitrary set of constraint edges $\t$ such that no edge from $\t$ intersects $R$.
   For any triangulation $\sigma\in\O^\t(n,k)$, 
   let $C_R(\sigma)$ be the number of disjoint top-to-bottom crossings of unit verticals from $\si$ that are inside $R$.
   Then, 
   $$
     \pr{C_R(\sigma) \leq \delta\, m}\leq e^{-c\, m}.
   $$
   Furthermore, let $\sigma,\sigma'$ be two triangulations sampled from the stationary distribution 
   $\mu$ given two different sets of constraint edges $\t,\t'$ such that no edge of $\t,\t'$  intersects $R$. 
   Then, there exists a coupling of  $\sigma,\sigma'$ such that the probability that they have less than $\delta \,m$ common top-to-bottom crossings of unit verticals 
   is at most $e^{-c\, m}$.
\end{lemma}

%############################################################################################
%############################################################################################
%############################################################################################
\section{Estimates via canonical paths}\label{sec:canonical}
We recall that the relaxation time $\trel$ is defined as the inverse of the spectral gap of the Markov chain. 
We start by showing that a direct application of the usual canonical path argument~\cite{AlSi} yields an exponential bound on the relaxation time of the Markov chain that is valid for all $\l\leq 1$. 
We recall the well known estimate relating $\trel$ and $\tmix$ (see, e.g., \cite[Theorem 12.3]{LPW}):
   \begin{equation}
\tmix\leq \trel(2+\log(1/\mu_*)),
      \label{trelmustart}
   \end{equation}
   where $\mu_*=\min_{\sigma} \mu(\sigma)$.
%One can observe that, despite its simplicity, the following estimate is already enough to distinguish the $\l\leq 1$ regime from the super-critical regime $\l>1$ for thin rectangles. Indeed,  we know that for $\l>1$ one has  $\trel\geq \exp(\O(n^2/k))$ for all $n\geq k^2$; see \cite[Theorem 6.2]{CMSS14}. %
%Structural property: it follows from the property recalled above that  given two triangulations $\si,\eta$ compatible with the constraint edges one can find a path in the flip graph from $\si$ to $\eta$ such that...
%
%
%Here we prove an exponential bound on the relaxation time for arbitrary $m \times n$ lattice triangulations and arbitrary constraints 
%for all $\lambda \leq 1$.{\marginpar{\tiny Check that holds for $\lambda=1$}}
\begin{theorem}\label{thm:expmixing}
  There exists a positive constant $C$ such that for any $\lambda\leq 1$, $n,k\in\bbN$ and any set of constraint edges $\t$, the Glauber dynamics on $\O^\t(n,k)$  
  satisfies
   $$
      T_\rel \leq \exp(Ckn).
   $$
\end{theorem}

Before proving the above theorem we recall a useful structural fact. Given a set of constraint edges $\t$ and a midpoint $x$, consider the set $\O^\t_x$ of possible values of $\si_x$, as $\si$ ranges in $\O^\t(n,k)$. Two edges $\si_x, \si'_x\in \O^\t_x$ are said to be {\it neighbors\/} if $\si_x$ is flippable
to $\si'_x$ within some triangulation $\si\in\O^\t(n,k)$. Then it is known (see, e.g., \cite{CMSS14})
that the induced graph with vertex set $\O^\t_x$ is a tree $\cG^\t_x$. We will make use of the following technical lemma; see \cite[Proposition~3.8]{CMSS14} for the proof.
\begin{lemma}\label{lem:distances}
Fix a set of constraint edges $\t$. For any midpoint $x$ and any two triangulations $\si,\si'\in\O^\t(n,k)$, the distance between $\sigma$ and $\si'$ in the flip graph  is equal to 
   $\sum_{x\in\Lkn} \kappa(\sigma_x,\si'_x)$, where $\kappa(\sigma_x,\si'_x)$ is the distance between $\sigma_x$ and $\si'_x$ in the tree $\cG^\t_x$.
\end{lemma}

\begin{proof}[{\bf Proof of Theorem~\ref{thm:expmixing}}]
  % Let $\Omega$ be the set of triangulations satisfying the constraints.
   For each pair $\sigma,\sigma'\in\O^\t(n,k)$, let $\Gamma_{\sigma,\sigma'}$ be a shortest path between $\sigma$ and $\sigma'$ in the flip graph.
   From Lemma~\ref{lem:distances}, we have that for any triangulation $\eta$ in the path $\Gamma_{\sigma,\sigma'}$ and any midpoint $x$,
   \begin{equation}
      |\eta_x| \leq |\sigma_x| \lor |\sigma'_x|.
%      \quad\text{and}\quad
%      |\eta_x| \geq |\sigma_x| \land |\sigma'_x|.
      \label{eq:monotonicity}
   \end{equation}
We can also assume that $\Gamma_{\sigma,\sigma'}$ is a {\it monotone\/} path in the sense that it is
composed of a sequence of edge-decreasing flips followed by a sequence of edge-increasing flips.
   
   Now, for any function $f\colon\Omega\to\mathbb{R}$,
   we have
   $$
   f(\sigma)-f(\sigma')=\sum_{(\eta,\eta') \in \Gamma_{\sigma,\sigma'}}\nabla_{\eta,\eta'}f,
   $$
    where we employ the notation $\nabla_{\eta,\eta'}f = f(\eta)-f(\eta')$. For simplicity, below we write $\mu$ instead of $\mu^\t$ and $\O$ instead of $\O^\t(n,k)$. 
    Thus, using Cauchy-Schwarz, the variance of $f$ with respect to $\mu$ satisfies 
   \begin{align}
      \var(f) 
      &= \frac{1}{2}\sum_{\sigma,\sigma'}\mu(\sigma)\mu(\sigma')(f(\sigma)-f(\sigma'))^2\nonumber\\
      &\leq \frac{1}{2}\,\cC(\O)\sum_{\eta,\eta'\colon \eta\sim\eta'}\mu(\eta)p(\eta,\eta')(\nabla_{\eta,\eta'}f )^2, \label{eq:ubvar}
   \end{align}
   where 
   $p(\eta,\eta')$ is the probability that the Glauber chain goes from $\eta$ to $\eta'$ in one step,
   $\eta\sim\eta'$ denotes that $\eta$ and $\eta'$ are adjacent triangulations,
   and 
   we use the notation 
 \begin{align}
      \cC(\O)= \max_{\eta,\eta'\colon \eta\sim\eta'}\sum_{\sigma,\sigma'\colon (\eta,\eta') \in \Gamma_{\sigma,\sigma'}}\frac{\mu(\sigma)\mu(\sigma')}{\mu(\eta)p(\eta,\eta')}|\Gamma_{\sigma,\sigma'}|,
       \label{cong}
   \end{align}
   for the so-called ``congestion ratio."
   Now assume that $p(\eta,\eta')\geq p(\eta',\eta)$, otherwise use reversibility to write $\mu(\eta)p(\eta,\eta')$ as $\mu(\eta')p(\eta',\eta)$.
   With this assumption we have that $p(\eta,\eta')\geq \frac{1}{2|\Lkn|}$. Also, from Lemma \ref{lem:distances} we have $|\Gamma_{\sigma,\sigma'}|= O(nk(n+k))$.
   The key property we use is that~\eqref{eq:monotonicity} gives
   $$
      \frac{\mu(\sigma)\mu(\sigma')}{\mu(\eta)}
      = Z^{-1}\prod_x \lambda^{|\sigma_x|+|\sigma'_x|-|\eta_x|} 
      \leq Z^{-1} \prod_x \lambda^{|\sigma_x| \land |\sigma'_x| }
      \leq 1,
         $$
         where we used the bound $$Z\geq \prod_x \lambda^{|\bar \sigma_x|  }
         \geq \prod_x\lambda^{|\sigma_x| \land |\sigma'_x| }.$$
   Plugging this into~\eqref{cong}, we obtain
    \begin{align}\label{cong2}
\cC(\O)\leq  Cnk(n+k)\,|\Lkn|\,|\O^\t(n,k)|^2.
    \end{align}
   Using Anclin's bound \cite{Anclin:2003jb} one has $|\O^\t(n,k)|\leq 2^{| \Lkn|}$. The proof is then concluded by recalling that $\trel$ is the smallest constant $\gamma$ such that the inequality 
   $$
   \var(f)\leq \frac{\gamma}2\sum_{\eta,\eta'\colon \eta\sim\eta'}\mu(\eta)p(\eta,\eta')(\nabla_{\eta,\eta'}f )^2
   $$
   holds for all functions $f:\O^\t(n,k)\mapsto\bbR$. 
\end{proof}

%############################################################################################
%############################################################################################
%############################################################################################
\subsection{An improved canonical paths argument}\label{sec:polymixing}
Here we establish a first polynomial bound on the relaxation time.  
The result here can be formulated as follows. 

\begin{theorem}\label{thm:polybound}
Fix $\l\in(0,1)$ and $k\in\bbN$.    There exists a positive constant $c=c(\lambda,k)$ such that for any boundary condition $\t=\{\t_x\}$ such that  $|\t_x|\leq n/4$ for all $x$, the relaxation time of the Glauber chain in $\O^\t(n,k)$
   satisfies  $$\trel\leq n^c.$$
\end{theorem}
The strategy of the proof is as follows.  
We shall identify a subset $\O'$ of triangulations such that the congestion ratio $\cC(\O')$ defined as in \eqref{cong} but restricted to $\O'$ satisfies a polynomial bound, in contrast with the exponential bound in \eqref{cong2}. Using a key input from \cite{AS}, we show that the Glauber chain enters the set $\O'$ with large probability after a {\em burn-in} time of $T=O(n^2)$ steps. Following an idea already used in \cite{CLMST} we establish the desired upper bound on $\trel$ by combining the above facts. 

We start with a deterministic estimate. 
\begin{lemma}\label{lengthle}
   Let $\sigma\in\O^\t(n,k)$ be a triangulation of the $n\times k$ rectangle with  boundary condition $\t=\{\t_x\}$ such that 
   $|\t_x|\leq L$ for all $x$.
   Then, all edges of $\sigma$ are contained in the rectangle $[-L,n+L]\times [-L,k+L]$.
\end{lemma}
\begin{proof}
   First, note that the ground state triangulation must satisfy the lemma, because all edges have size at most $L$. 
   Now it is enough to show that there cannot be an increasing edge $\sigma_x$ with $x\in\Lambda_{n,k}$ 
   such that $\sigma_x^x \not\subset [-L,n+L]\times[-L,k+L]$ but all edges of $\sigma$ are inside $[-L,n+L]\times [-L,k+L]$. 
   We use the notation $\sigma^x$ to denote the triangulation obtained from $\sigma$ by flipping $\sigma_x$.
   In order to achieve a contradiction, 
   assume that such an increasing edge $\sigma_x$ exists and assume that $\sigma_x^x$ is at the left part of the triangulation (i.e., that 
   its leftmost endpoint has horizontal coordinate smaller than $-L$).
   Let $\sigma_y,\sigma_z$ be the triangle containing $\sigma_x$ such that the vertex $v=\sigma_y\cap\sigma_z$ has horizontal coordinate smaller than
   $-L$. 
   Since $\sigma$ is completely inside $[-L,n+L]\times [-L,k+L]$, we obtain that $\si_y$ and $\si_z$ are constraint edges. 
   Also, since $x\in\Lambda$, $\sigma_x$ must have one endpoint $u$ of horizontal coordinate at least $0$.
   This gives that $\|v-u\|_1>L$, and consequently, either $\sigma_y$ or $\sigma_z$ has length larger than $L$, 
   which is a contradiction.
\end{proof}
Next, we formulate a general upper bound on $\trel$ in terms of the congestion ratio of a subset $\O'$ of the state space $\O$, a time $T$, and the probability needed to reach $\O'$ within time $T$. A version of this lemma appears in \cite[Theorem 2.4]{CLMST}. For the reader's convenience we give a detailed proof. 

\begin{lemma}[Canonical paths with burn-in time]\label{lem:canonicalburnin}
   Consider a Markov chain with state space $\O$, irreducible transition matrix $p(\cdot,\cdot)$ and reversible probability measure $\mu$. Let $\Omega'\subset\Omega$ be a subset so that between each 
   $\si,\si'\in\Omega'$ there is a path $\Gamma_{\si,\si'}$ in the Markov chain that is entirely contained in $\Omega'$. Define the congestion ratio 
   \begin{align}\label{congr1}
      \cC(\O') = \max_{\eta,\eta'\in\Omega'\colon \eta\sim\eta'} \sum_{\si,\si'\colon (\eta,\eta')\in \Gamma_{\si,\si'}}
         \frac{\mu(\si)\mu(\si')|\Gamma_{\si,\si'}|}{\mu(\eta)p(\eta,\eta')},
   \end{align}
   where the sum is over all pairs of states $\si,\si'\in\O'$ so that the path $\Gamma_{\si,\si'}$ 
   uses the transition $(\eta,\eta')$.
   Fix $T\in\bbN$ and let $\rho$ be a lower bound on the probability that at time $T$ the chain is inside $\Omega'$, uniformly over the starting state in $\Omega$. Then
   the relaxation time satisfies
   $$
      T_\mathrm{rel} \leq \frac{6\,T^2}{\rho} + \frac{3\,\cC(\O')}{\rho^2}.
   $$
\end{lemma}
\begin{proof}
   We run the Markov chain for $T$ steps.
   For $\sigma,\tau\in\Omega$, let 
   $\mu_\sigma(\tau)$ be the probability that, starting from $\sigma$, the Markov chain is at $\tau$ after $T$ steps. 
   Note that $\mu_\sigma(\Omega')\geq \rho$.
   For $\sigma,\tau\in\Omega$, and for any path~$\gamma$ of length~$T$ in the chain
   starting at~$\sigma$ and ending at~$\tau$,
   %$\gamma=(\gamma_0,\gamma_1,\ldots,\gamma_T)$ 
%   a path of $T+1$ adjacent triangulations with $\gamma_0=\sigma$ and $\gamma_T=\tau$, 
   let $\nu_{\sigma,\tau}(\gamma)$ be the conditional probability that, 
given the initial state $\sigma$ at time $0$ and the final state $\tau$ after $T$ steps, the Markov chain traverses the path $\gamma$. 
   Then, for any function $f\colon\Omega\to\mathbb{R}$, we have 
   \begin{align*}
      &\var(f) = \frac{1}{2}\sum_{\sigma,\sigma'\in\Omega} \mu(\sigma)\mu(\sigma')(f(\sigma)-f(\sigma'))^2
      = \frac{1}{2}\sum_{\sigma,\sigma'\in\Omega} \sum_{\eta,\eta'\in\Omega'}\mu(\sigma)\mu(\sigma')\frac{\mu_\sigma(\eta)\mu_{\sigma'}(\eta')}{\mu_\sigma(\Omega')\mu_{\sigma'}(\Omega')}\times \\ 
       & \quad \quad\times\sum_{\gamma_1,\gamma_2}\nu_{\sigma,\eta}(\gamma_1)\nu_{\sigma',\eta'}(\gamma_2)
       \textstyle{ \Big(\sum_{e\in\gamma_1}\nabla_e f + \sum_{e\in\gamma_2}\nabla_e f + \sum_{e\in\Gamma_{\eta,\eta'}}\nabla_e f\Big)^2},
   \end{align*}
   where the three sums inside the parenthesis are over the edges of the paths $\gamma_1,\gamma_2,$ and $\Gamma_{\eta,\eta'}$, respectively.
   Then, applying Cauchy-Schwarz, we obtain
   \begin{align*}
%   &   \sum_{\sigma,\sigma'\in\Omega} \mu(\sigma)\mu(\sigma')(f(\sigma)-f(\sigma'))^2
& \var(f)
      \leq \frac{3}{2}\sum_{\sigma,\sigma'\in\Omega} \sum_{\eta,\eta'\in\Omega'}\mu(\sigma)\mu(\sigma')\frac{\mu_\sigma(\eta)\mu_{\sigma'}(\eta')}{\mu_\sigma(\Omega')\mu_{\sigma'}(\Omega')}\times \\ & \quad \quad\times
         \sum_{\gamma_1,\gamma_2}\nu_{\sigma,\eta}(\gamma_1)\nu_{\sigma',\eta'}(\gamma_2)       \textstyle{ \left (T\sum_{e\in\gamma_1}(\nabla_e f )^2 + T\sum_{e\in\gamma_2}(\nabla_e f )^2 + |\Gamma_{\eta,\eta'}|\sum_{e\in\Gamma_{\eta,\eta'}}(\nabla_e f )^2\right)}.
   \end{align*}
   We write the right-hand side above as $A_1+A_2+A_3$, where
   \begin{align*}
      A_1 & = \frac{3}{2}\sum_{\sigma,\sigma'\in\Omega} \sum_{\eta,\eta'\in\Omega'}\mu(\sigma)\mu(\sigma')\frac{\mu_\sigma(\eta)\mu_{\sigma'}(\eta')}{\mu_\sigma(\Omega')\mu_{\sigma'}(\Omega')}
         \sum_{\gamma_1,\gamma_2}\nu_{\sigma,\eta}(\gamma_1)\nu_{\sigma',\eta'}(\gamma_2)
         \,T\sum_{e\in\gamma_1}(\nabla_e f )^2\\
      A_2 & = \frac{3}{2}\sum_{\sigma,\sigma'\in\Omega} \sum_{\eta,\eta'\in\Omega'}\mu(\sigma)\mu(\sigma')\frac{\mu_\sigma(\eta)\mu_{\sigma'}(\eta')}{\mu_\sigma(\Omega')\mu_{\sigma'}(\Omega')}
         \sum_{\gamma_1,\gamma_2}\nu_{\sigma,\eta}(\gamma_1)\nu_{\sigma',\eta'}(\gamma_2)
         \,T\sum_{e\in\gamma_2}(\nabla_e f )^2\\
      A_3 & = \frac{3}{2}\sum_{\sigma,\sigma'\in\Omega} \sum_{\eta,\eta'\in\Omega'}\mu(\sigma)\mu(\sigma')\frac{\mu_\sigma(\eta)\mu_{\sigma'}(\eta')}{\mu_\sigma(\Omega')\mu_{\sigma'}(\Omega')}
         \sum_{\gamma_1,\gamma_2}\nu_{\sigma,\eta}(\gamma_1)\nu_{\sigma',\eta'}(\gamma_2)
         \,|\Gamma_{\eta,\eta'}|\sum_{e\in\Gamma_{\eta,\eta'}}(\nabla_e f )^2.
   \end{align*}
   We start with $A_1$. Summing over $\gamma_2,\sigma',\eta'$, and using $\sum_{\gamma_2}\nu_{\si',\eta'}(\gamma_2)=1$, we have
   $$
      A_1
      = \frac{3}{2}T \sum_{\sigma\in\Omega} \sum_{\eta\in\Omega'}\mu(\sigma)\frac{\mu_\sigma(\eta)}{\mu_\sigma(\Omega')}\sum_{\gamma_1}\nu_{\sigma,\eta}(\gamma_1)
        \sum_{e\in\gamma_1}(\nabla_e f )^2.
   $$
   Changing the order of the summations, and summing first over all pairs of adjacent states $\tau\sim\tau'$, we get
   \begin{align*}
      A_1
      &=\frac{3}{2}T\sum_{\tau,\tau'\in\Omega\colon\tau\sim\tau'}\mu(\tau)p(\tau,\tau')(\nabla_{\tau,\tau'}f  )^2
         \sum_{\sigma\in\Omega,\eta\in\Omega',\gamma \colon (\tau,\tau') \in \gamma}\frac{\mu(\sigma)\mu_{\sigma}(\eta)\nu_{\sigma,\eta}(\gamma)}{\mu_\sigma(\Omega') \mu(\tau)p(\tau,\tau')}\\
      &\leq \frac{3T}{2\rho}\sum_{\tau,\tau'\in\Omega\colon\tau\sim\tau'}\mu(\tau)p(\tau,\tau')(\nabla_{\tau,\tau'}f  )^2
         \sum_{\sigma\in\Omega,\eta\in\Omega',\gamma \colon (\tau,\tau') \in \gamma}\frac{\mu(\sigma)\mu_{\sigma}(\eta)\nu_{\sigma,\eta}(\gamma)}{\mu(\tau)p(\tau,\tau')}\\
      &\leq \frac{3T}{2\rho}\sum_{\tau,\tau'\in\Omega\colon\tau\sim\tau'}\mu(\tau)p(\tau,\tau')(\nabla_{\tau,\tau'}f  )^2
         \frac{\bbP_\mu\left({\text{Markov chain traverses $(\tau,\tau')$ within $T$ steps}}\right)}{\mu(\tau)p(\tau,\tau')}\\
      &\leq \frac{3T}{2\rho}\sum_{\tau,\tau'\in\Omega\colon\tau\sim\tau'}\mu(\tau)p(\tau,\tau')(\nabla_{\tau,\tau'}f  )^2
         \frac{T\mu(\tau)p(\tau,\tau')}{\mu(\tau)p(\tau,\tau')}
      = \frac{3T^2}{\rho}\cD(f,f),
   \end{align*}
   where $\bbP_\mu(\cdot)$ denotes the measure induced by the Markov chain started from stationarity, and we use the notation
   \begin{equation}\label{dirich}
      \cD(f,f) = \frac12\sum_{\tau,\tau'\in\Omega\colon\tau\sim\tau'}\mu(\tau)p(\tau,\tau')(\nabla_{\tau,\tau'}f  )^2
   \end{equation}
   for the so-called Dirichlet form. 
   For the second term, we have by symmetry that $A_2=A_1$. 
   For $A_3$, we use $\rho\leq \mu_\sigma(\Omega'),\mu_{\sigma'}(\Omega')$, and sum over $\gamma_1,\gamma_2,\sigma,\sigma'$ to obtain
   \begin{align*}
      A_3
      \leq\frac{3}{2\rho^2}\sum_{\eta,\eta'\in\Omega'}\mu(\eta)\mu(\eta')|\Gamma_{\eta,\eta'}|\sum_{e\in\Gamma_{\eta,\eta'}}(\nabla_e f )^2.
   \end{align*}
   Changing the order of summations, we get
   \begin{align*}
      A_3
%      \leq\frac{3}{\rho^2}\sum_{\tau,\tau'\in\Omega'\colon\tau\sim\tau'}\mu(\tau)p(\tau,\tau')(\nabla_{\tau,\tau'}f  )^2
%         \sum_{\eta,\eta'\in\Omega'\colon (\tau,\tau')\in \Gamma_{\eta,\eta'}}\frac{\mu(\eta)\mu(\eta')|\Gamma_{\eta,\eta'}|}{\mu(\tau)p(\tau,\tau')}
      \leq \frac{3\,\cC(\O')}{\rho^2}\cD(f,f).
   \end{align*}
The result now follows since  $\trel$ is the smallest constant $\gamma$ such that the inequality 
   $$
   \var(f)\leq \gamma\,
   \cD(f,f)  $$
   holds for all functions $f:\O\mapsto\bbR$. 
\end{proof}

\begin{proof}[{\bf Proof of Theorem \ref{thm:polybound}}]
   Let $T=c_1n^2k$ for some large enough constant $c_1=c_1(\lambda)>0$.
   Thanks to Lemma \ref{lengthle} we may apply Lemma~\ref{lem:crosstail} 
   with $M=2n+k$. Thus, for any given $x\in\Lkn$ and ground-state edge $\bar\sigma_x$ with midpoint $x$, taking $\ell=c_2 \log |\Lkn|$ for some large enough constant $c_2=c_2(\lambda)>0$, 
   and taking the union bound over all $x\in\Lkn$ in \eqref{grcr} we obtain that the triangulation $\si^T$ at time $T$,  for an arbitrary initial condition $\si$, satisfies  
   \begin{equation}
      \pr{\bigcup\nolimits_{x\in\Lkn}\bigcup\nolimits_{y\in\Lkn}\big\{\sigma_y^T \cap \bar \sigma_x\neq\emptyset\big\}\cap \big\{|\sigma_y^T| > |\bar \sigma_x|+\ell\big\}}
      \leq n^{-1}.
      \label{eq:goodset}
   \end{equation}
   Let 
   $$
      \Omega'
      = \Big\{\sigma \colon \text{ for all $x,y\in\Lkn$,} \,
         |\sigma_x|\leq |\bar\sigma_x|+\ell
         \text{ and }\ind{\sigma_y\cap\bar\sigma_x\neq\emptyset} \leq \ind{|\sigma_y|\leq |\bar\sigma_x|+\ell}\Big\}.
   $$ 
   Thus~\eqref{eq:goodset} implies that $\pr{\sigma^T\in\Omega'}\geq 1-n^{-1}$.
   Note that $\Omega'$ is a decreasing set in the sense that if $\sigma \in \Omega'$ then for all $\sigma'$ that can be obtained from~$\sigma$ by 
   performing decreasing flips, we have $\sigma'\in\Omega'$.
   This allows us to construct a path $\Gamma_{\si,\si'}$ within $\Omega'$ between any pair of triangulations $\sigma,\sigma'\in\Omega'$.
   
   We now describe the path $\Gamma_{\si,\si'}$. 
   Fix two triangulations $\sigma,\sigma'\in\Omega'$, and any midpoint $x\in\Lkn$.
   Let $g$ be a ground state edge at~$x$.
   The edges that need to be flipped to transform $\sigma_x$ into 
   $\sigma_x'$ are contained in $I_{g}(\sigma,\ell)\cup I_{g}(\sigma',\ell)$ (recall the definition of $I_g$ from Lemma~\ref{lem:setI}).
   By Lemma~\ref{lem:setI} we have that all edges in $\bigcup_{\sigma\in\Omega'}I_{g}(\sigma,\ell)$ have midpoint inside a ball of radius $2\ell$
   centered at $x$. This implies that if we partition $[0,n]\times[0,k]$ into slabs of horizontal width $2\ell$, we can find a sequence of 
   flips that transform $\sigma$ into $\sigma'$ slab by slab, from left to right, so that when transforming the $i$th slab, only edges with midpoints
   in the $i$th and $(i+1)$th slabs need to be flipped. In each slab, we just perform the minimum number of flips needed to transform that slab into
   $\sigma'$, and we do that by first performing all decreasing flips and then all increasing flips. 
   
   Our goal is to apply
   Lemma~\ref{lem:canonicalburnin}, for which we need to bound  the value of the congestion 
   ratio $\cC(\O')$. To do this, consider a pair of adjacent triangulations $\eta,\eta'$.
   Assume that $\eta,\eta'$ differ at an edge of the $i$th slab. Therefore, if $\si,\si'$ are two triangulations for which the path between 
   them includes the transition $(\eta,\eta')$ we know that triangulation $\eta$ has slabs $1,2,\ldots,i-2$
    equal to $\si'$ and slabs $i+2,i+3,\ldots$ equal to 
   $\si$.
   Let $\xi$ be a partial triangulation in $\Omega'$ of the first $i-2$ slabs
   and $m$ be a partial triangulation in $\Omega'$ of the middle slabs so that $\xi$, $m$ and $\si$ are compatible, meaning
   that $\xi$, $m$ and the edges of $\si$ inside slabs $i+2,i+3,\ldots$ can coexist to form a full triangulation. 
   Similarly, let $\xi'$ be a partial triangulation in $\Omega'$ of the last slabs ($i+2,i+3,\ldots$)
   and $m'$ be a partial triangulation in $\Omega'$ of the middle slabs so that $\xi'$, $m'$ and the edges of $\si'$ inside 
   slabs $1,2,\ldots,i-2$ are compatible.  
   Assume that $p(\eta,\eta')\geq p(\eta',\eta)$, which implies that $p(\eta,\eta')\geq \frac{1}{2|\Lkn|}$ 
   (otherwise, replace $\mu(\eta)p(\eta,\eta')$ with $\mu(\eta')p(\eta',\eta)$ in $\cC(\O')$).
   Let $\eta_i$ be the part of $\eta$ inside slabs $i-1,i,i+1$.
   Then, summing over all $\xi,\xi',m,m'$ as above such that $(\eta,\eta')$ is a transition in the path from
   $\xi,m,\si$ to $\si',m',\xi'$, and noting that the path between $\si$ and $\si'$ has length at most $2\ell|\Lkn|$, 
   we obtain the following upper bound for $\cC(\O')$:
   $$
      \cC(\O') \leq 4\ell|\Lkn|^2\sum_{\xi,\xi',m,m'} \frac{\lambda^{|\xi|+|\xi'|+|m|+|m'|-|\eta_i|}}{Z_{\Omega'}},
   $$
   where $Z_{\Omega'}=\sum_{\sigma\in\Omega'}\lambda^{|\sigma|}$.
   Instead of summing over $m,m'$, we will sum over triangulations $m''$ of the middle slabs that are compatible with 
   both $\xi$ and $\xi'$ and are to be interpreted as $m \land m'$. 
   Given $m''$, we sum over $m,m'$ that can be obtained from $m''$ by increasing flips
   and such that $(\eta,\eta')$ is a transition in the path from 
   $\xi,m,\si$ to $\si',m',\xi'$.
   Let $A(m'',m,m',\eta)$ be the indicator that all four of them are compatible, as described above. 
   When $A(m'',m,m',\eta)=1$ we have that
   $|\eta_x|\leq |m_x|\lor |m'_x|$ for any midpoint $x$ in the middle slabs.
   Hence, $|m''|+|m\setminus m''|+|m'\setminus m''|\geq|\eta_i|$, which gives
$$   
%   \begin{align*}
      \cC(\O') 
      %&\leq 4\ell|\Lkn|^2 \sum_{\xi,\xi',m,m'} \frac{\lambda^{|\xi|+|\xi'|+|m|+|m'|-|\eta_i|}}{Z_{\Omega'}}\\
        \leq 4\ell|\Lkn|^2 \sum_{\xi,\xi',m''} \frac{\lambda^{|\xi|+|\xi'|+|m''|}}{Z_{\Omega'}}
         \sum_{m,m'}\lambda^{|m''|+|m\setminus m''|+|m'\setminus m''|-|\eta_i|}A(m'',m,m',\eta).
%   \end{align*}
$$
Since $\lambda<1$, we can simply use Anclin's bound \cite{Anclin:2003jb} saying that the number of 
triangulations of an $\ell\times k$ region with arbitrary constraint edges is at most $2^{3 k \ell}$ to obtain that 
   $$
      \cC(\O') \leq 4\ell|\Lkn|^2 2^{6k\ell} \sum_{\xi,\xi',m''} \frac{\lambda^{|\xi|+|\xi'|+|m''|}}{Z_{\Omega'}}
        \leq 4\ell|\Lkn|^2 2^{6k\ell}.
   $$
   Plugging everything into Lemma~\ref{lem:canonicalburnin} completes the proof.
\end{proof}

%############################################################################################
%############################################################################################
%############################################################################################
\section{Proof of Theorem \ref{mainth}}\label{sec:tightmixing}
\subsection{High-level overview}\label{sec:poverview}
The proof is composed of three main ingredients: (i) a good ensemble, (ii) a 
decay of correlation analysis, and (iii) a recursion for the logarithmic Sobolev inequality.

\emph{The good ensemble.}
The first step is to show that uniformly over the initial condition, with high probability, for all times $t\in[T,T+n^{2}]$, with $T=O(n^2)$, 
the Markov chain stays within 
a subset %we can restrict our attention to a subset 
$\tilde\O$ of triangulations where all edges have length at most
$C \log n$ for some constant $C>0$. We will call this subset the \emph{good ensemble}. 
This result will be a consequence of  the tail estimate of Lemma~\ref{lem:tailedge}.
Therefore, we will couple our evolution in the time interval $t\in[T,T+n^{2}]$ with the Markov chain
restricted to the good ensemble, which evolves 
as before, by attempting to flip edges chosen uniformly at random, but with the suppression of any edge flip that would render an edge longer than $C\log n$. 
The structural properties of triangulations imply that this Markov chain is irreducible. 
Moreover, the reversible probability measure is given by $\tilde \mu=\mu(\cdot\mid \tilde\O)$, the measure $\mu$ conditioned on the event $\si\in\tilde \O$.  
Since $\mu$ and $\tilde \mu$ can be coupled with high probability, it is sufficient to analyze convergence to equilibrium for the restricted chain, 
and to show that the latter mixes in time $T'=O(n^2)$. We will actually prove that the restricted chain mixes in time $T'=n \polylog(n)$. 
For the rest of this discussion we assume that we are working with the Markov chain restricted to the good ensemble $\tilde\O$. 
%For simplicity of notation we still write $\mu$ for the conditional probability on $\tilde \O$ 

\emph{Decay of correlations.} We split 
the set of midpoints $\Lkn$ into two intersecting slabs $\Lambda_\ml$ and $\Lambda_\mr$, where $\Lambda_\ml$ contains all
midpoints with horizontal coordinate smaller than $n/2+2C\log n$ and $\Lambda_\mr$ contains all midpoints with horizontal coordinate at least 
$n/2-2C\log n$. Note that $\Lambda_\ml\cap\Lambda_\mr$ is a slab of height $k$ and horizontal width $4C\log n$. 
Let $\mathcal{F}_\mr, \mathcal{F}_\ml$ be the $\sigma$-algebras generated by the edges with
midpoints in $\Lambda_\mr\setminus \Lambda_\ml, \Lambda_\ml\setminus \Lambda_\mr$ 
respectively.
We want to show that, conditional on any event $F\in\cF_\mr$, the distribution of  
%collection of edges with midpoints 
%in 
%%the leftmost part of the slab 
%$\Lambda_\ml\setminus \Lambda_\mr$, the distribution of 
the edges in %the rightmost part (
$\Lambda_\ml\setminus \Lambda_\mr$ is not affected much, and similarly for events $F\in\cF_\ml$. The intuition for this is that the intersection $\Lambda_\ml \cap \Lambda_\mr$ of the slabs 
is large enough to allow correlations from $\Lambda_\ml\setminus \Lambda_\mr$ to decay. 
We will make this intuition rigorous by showing that 
there exists a positive $\epsilon=\epsilon(\l)$ such that, for all $\cF_\ml$-measurable functions $f_\ml$ and  all $\cF_\mr$-measurable functions
$f_\mr$, we have
\begin{equation}
   \sup_{F\in\mathcal{F}_\mr}\big|\tilde\mu(f_\ml \mid F) -\tilde \mu(f_\ml)\big| \leq n^{-\epsilon} \|f_\ml\|_1
   \quad\text{and}\quad
   \sup_{F\in\mathcal{F}_\ml}\big|\tilde\mu(f_\mr \mid F) - \tilde\mu(f_\mr)\big| \leq n^{-\epsilon} \|f_\mr\|_1,
   \label{eq:decayhighlevel}
\end{equation}
where $\tilde \mu(f\mid F)$ stands for the expectation of $f$ given the event $F$ and we use $\|f\|_1$ to denote the $L^1$ norm $\|f\|_1=\sum_{\si\in\tilde\O}\tilde\mu(\si)|f(\si)|$. 

The high-level argument for~\eqref{eq:decayhighlevel} is the following.
Fix any valid collection of edges with midpoints in $\Lambda_\ml\setminus \Lambda_\mr$, that is, a partial triangulation from $\tilde\O$. This defines an event $F\in\cF_\ml$.
We will construct a coupling of one triangulation $\sigma$ distributed according to $\tilde\mu(\cdot \mid F)$ and another triangulation $\sigma'$ distributed 
according to $\tilde \mu(\cdot)$. We do this by first sampling the edges of $\sigma'$ whose midpoint is in $\Lambda_\ml\setminus \Lambda_\mr$. 
Call this event $F'\in\cF_\ml$. Since we are restricted to the good ensemble, the edges of $F$ and $F'$
have length at most $C\log n$. Therefore, 
none of them crosses into the right half of $\Lambda_\ml\cap\Lambda_\mr$. 
Lemma~\ref{lem:toptobottom} therefore ensures that we may couple the sampling of
edges in $\Lambda_\mr$ so that, with probability at least $1-e^{-\epsilon \log n}$, we put the same top-to-bottom crossing of unit verticals
in $\sigma$ and $\sigma'$ inside 
the right half of $\Lambda_\ml\cap\Lambda_\mr$. In particular, this implies that we can couple $\sigma$ and $\sigma'$ so that 
they agree on $\Lambda_\mr\setminus\Lambda_\ml$. This will establish~\eqref{eq:decayhighlevel}.

\emph{The log-Sobolev inequality.} 
An important ingredient in the proof of Theorem \ref{mainth} is the use of the logarithmic Sobolev inequality for the good ensemble.   
For any positive function $f$, 
let $\tilde \mu(f)$ stand for the expectation of $f$ in the good ensemble, and let 
$$
   \ent(f)
   =\tilde \mu\left(f\log \left(\tfrac{f}{\tilde\mu(f)}\right)\right)
   =\sum_{\sigma} \tilde \mu(\sigma)f(\sigma)\log \left(\tfrac{f(\sigma)}{\tilde\mu(f)}\right)
$$ 
denote the entropy of $f$. 
Also, define 
$$
   \mathcal{E}(f,f)=\frac{1}{2}\sum_{\sigma,\sigma'\in\tilde\O}\tilde\mu(\sigma)\r(\si,\si')(f(\sigma)-f(\sigma'))^2,
$$
where 
$$
\r(\si,\si') = \frac{\l^{|\si'|}}{\l^{|\si|}+\l^{|\si'|}}\ind{\si\sim\si'}\,.
$$
As usual $\sigma\sim\sigma'$ means that $\si,\si'$ differ by a single edge flip.
Note that $\r(\si,\si')= |\Lkn|p(\si,\si')$, where $p$ is the transition matrix of the discrete time chain. 
%, and therefore  $\cE(f,f)= |\Lkn|\cD(f,f)$, where $\cD(f,f)$ is defined in \eqref{dirich}. 
Thus $\cE(f,f)$ can be interpreted as the Dirichlet form of the continuous time Markov chain where 
every edge of the triangulation independently attempts to flip at rate 1.

Let $\sobolev$ be the {\em log-Sobolev constant} of this Markov chain, defined as the smallest constant  $c>0$  such that for all 
functions $f$ one has
\begin{equation}
   \ent(f^2) \leq c\,\mathcal{E}(f,f).
   \label{eq:sobolev}
\end{equation}
It is known (see e.g.\ \cite[Theorem~2.9]{Fabio}) that $\sobolev$ is related to the mixing and relaxation times via 
\begin{equation}\label{sob_mix}
   \tilde T_\mix \leq \frac{\sobolev}{4}\left(4+ \log_+\log\tilde\mu_*^{-1}\right)
   \quad\text{and}\quad
   2\,\tilde  T_\rel \leq \sobolev \leq \tilde T_\rel \left(\frac{\log(\tilde\mu_*^{-1})}{1-2\tilde\mu_*}\right),
\end{equation}
where $\tilde\mu_*=\min_{\sigma\in\tilde \Omega} \tilde\mu(\sigma)$, and we use $\tilde  T_\rel, \tilde
          T_\mix$ to denote the relaxation time and the mixing time of the continuous time chain restricted to the good set. 
These bounds should be compared with \eqref{trelmustart}. In particular, it will be crucial for us to work with the log-Sobolev constant rather than the relaxation time in order to obtain the strong bound on mixing time claimed in Theorem~\ref{mainth}.
% Indeed, in the good ensemble one has $\mu_*^{-1}=poly(n)$ and therefore $\tilde T_\mix \leq \sobolev \polylog(n)$.  

\emph{Recursion.} We will bound the (restricted) log-Sobolev constant via the so-called bisection method introduced in \cite{Martinelli:1999jm}. 
Let $\Lambda_\ml,\Lambda_\mr$  and $\mathcal{F}_\ml,\mathcal{F}_\mr$ be as above. 
Using the decay of correlations in \eqref{eq:decayhighlevel}, the decomposition estimate in \cite[Proposition 2.1]{Cesi} implies that for all functions $f:\tilde\O\mapsto\bbR$ we have
\begin{equation}
   \ent(f^2)
   \leq \left(1+O(n^{-\epsilon})\right) \tilde\mu\left[\ent(f^2\mid \mathcal{F}_\ml) + \ent(f^2\mid \mathcal{F}_\mr)\right]
   \leq \left(1+O(n^{-\epsilon})\right) 2\,\sobolev^{(1)} \mathcal{E}(f,f),
   \label{eq:overview1}
\end{equation}
where $\sobolev^{(1)}$ is the largest log-Sobolev constant among the systems conditioned on $\mathcal{F}_\ml$ and $\mathcal{F}_\mr$ and the factor $2$ comes from the double counting of flips within the region $\L_\ml\cap\L_\mr$.
Hence, we obtain that $\sobolev \leq \left(1+O(n^{-\epsilon})\right)2\sobolev^{(1)}$.
We would then like to recursively apply the same strategy to bound $\ent(f^2\mid \mathcal{F}_\ml)$ and $\ent(f^2\mid \mathcal{F}_\mr)$. 
Indeed, $\tilde\mu(\cdot\mid \cF_\mr)$ is a Gibbs measure on triangulations with midpoints in $\L_\ml$, and  
we may split $\Lambda_\ml$ into two intersecting slabs, establish decay of correlations and again use the decomposition above to further reduce the original scale. 
One caveat is that now we have to take into account the boundary conditions dictated by the conditioning on $\cF_\mr$. 
These consist of constraint edges protruding from the right boundary, with midpoints in $\L_\mr\setminus\L_\ml$. 
The boundary conditions will not be a major problem since we are in the good ensemble so these edges cannot protrude more than a distance
$C\log n$.
After $j$ such iterations, we will be considering slabs of size roughly $n2^{-j}$, with edges of size at most $C\log n$ protruding from both the left and right boundaries. It will be convenient to iterate this procedure for 
$j=j_*$ steps, where  $n2^{-{j_*}}$ is roughly $\log^6 n$, so that protruding boundary edges are still far away from the middle of the slab, which is 
the crucial region for exploiting the decay of correlations. With this strategy, after $j_*$ iterations we obtain
$$
   \sobolev 
   \leq 
   \left(1+O(n^{-\epsilon})\right)^{j_*}2^{j_*}\sobolev^{(j_*)}.
$$
Employing the general polynomial bound on the 
relaxation time of Theorem~\ref{thm:polybound} and the relation between $\sobolev$ and $T_\rel$, 
we obtain that $\sobolev^{(j_*)}$ is at most $\polylog(n)$ uniformly over all boundary conditions in the good ensemble.
The main problem is that the term $2^{j_*}$ is 
too large (of order $\frac{n}{\log^6n}$ by our choice of $j_*$).
As in \cite{Martinelli:1999jm} we overcome this difficulty by {\it randomizing\/} the location of the split of $\Lkn$ into $\Lambda_\ml$ and $\Lambda_\mr$, and similarly for the other scales. The idea is to 
first split $\Lkn$ into three disjoint slabs with height $k$, the left and right slabs with horizontal length $\tfrac12(n-\log^3n)$, and the middle slab with horizontal  length $\log^3n$. 
Then we further split the middle slab into smaller slabs (that we call \emph{rectangles}) each with horizontal length $4C\log n$. 
We choose one such rectangle uniformly at random, and define $\Lambda_\ml$ to be the midpoints 
to the left of this rectangle (including the rectangle)
and $\Lambda_\mr$ to be the midpoints to the right of this rectangle (including the rectangle). With this randomization,~\eqref{eq:overview1} will be improved to
$$
   \ent(f^2)
   \leq \left(1+O(1/\log ^2 n)\right) \sobolev^{(1)} \mathcal{E}(f,f),
$$
where $\log ^2 n$ is roughly the number of rectangles in the middle slab of $\Lkn$.
Then, iterating $j_*$ times (with $j_*$ as above) we get
\begin{equation}\label{polylog}
   \sobolev
   \leq \left(1+O(j_*/\log^2n)\right) \sobolev^{(j_*)} = \polylog(n).
\end{equation}
Once we obtain \eqref{polylog}, using \eqref{sob_mix} we can conclude 
that the continuous time Markov chain restricted to the good ensemble satisfies $\tilde T_\mix = \polylog(n)$. From this the desired conclusion for the discrete time Glauber dynamics
will follow in a simple way. 

We now proceed with the detailed proof of Theorem \ref{mainth}. 
\subsection{The good ensemble}\label{sec:goodensemble}
Let $\sigma^0,\sigma^1,\ldots$ be the discrete time Markov chain on triangulations of $\Lokn$ with no constraint edges.
The first step is to show that after a burn-in time of order $n^2$, during a very long time interval, 
the largest edge of the triangulation is of order at most $\log n$. 
Let $C=C(\lambda)$ be a large enough constant, and define
\begin{equation}
   \tilde\Omega = \Big\{\sigma\in\Omega \colon |\sigma_x|\leq C\log n \text{ for all $x\in\Lkn$}\Big\}.
   \label{eq:goodensemble}
\end{equation}
The set $\tilde\Omega$ represents the good ensemble.
The next lemma will allow us to analyze the Markov chain restricted to the set $\tilde\Omega$.
\begin{lemma}\label{lem:burnin}
   Fix $\l\in(0,1)$. There exists a constant $c_1=c_1(\lambda)$ so that if we set $T=c_1 n^2$ then for all $n\geq k\geq 1$
   $$
      \pr{\bigcap\nolimits_{t=T}^{T+n^{2}} \big\{\sigma^t\in \tilde\Omega\big\}} \geq 1 - n^{-2}.
   $$
\end{lemma}
\begin{proof}
   For any given $x\in\Lkn$ and any $t\geq c_1n^2$, Lemma~\ref{lem:tailedge} gives that
   $$
      \pr{|\sigma_x^t| > C\log n} 
      \leq \exp(-c_2 C \log n),
   $$
   for some constant $c_2$ independent of $C$ and $n$.
   Setting $C$ large enough and taking a union bound over all $x\in\Lkn$ and all integers $t\in[T,T+n^{2}]$ concludes the proof.
\end{proof}

%############################################################################################

%############################################################################################
\subsection{Decay of correlations}\label{sec:decay}
Let $\Gamma\subset\Lambda$ be a {\it slab\/} of width $w$; that is, for some $x\in\mathbb{Z}$,
$$
   \Gamma = \Lkn \cap [x,x+w]\times[0,k].
$$
We assume throughout that $w\geq \tfrac12\,C^6\log^6 n$, where $C$ is fixed as in \eqref{eq:goodensemble}.

Partition $\Gamma$ into three slabs, two of width roughly $\tfrac12(w-C^3\log^3n)$ and one of width roughly $C^3\log^3n$. More precisely, for $\Gamma$ as above, let
$$
   \Gamma_1 = \Lkn \cap \big[x,x+\tfrac{w-C^3\log^3n}{2}\big]\times[0,k],
   \quad
   \Gamma_2 = \Lkn \cap \big(x+\tfrac{w-C^3\log^3n}{2},x+\tfrac{w+C^3\log^3n}{2}\big]\times[0,k]
$$
$$
   \text{and}\quad
   \Gamma_3 = \Lkn\cap \big(x+\tfrac{w+C^3\log^3n}{2},x+w\big]\times[0,k].
$$
Partition the middle slab $\Gamma_2$ into disjoint slabs $J_1,J_2,\ldots,J_s$ (from left to right) each of width $4C\log n$, with  
\begin{equation}
   s=\frac{C^3\log^3n}{4C\log n}=\frac{C^2\log^2n}{4}.
   \label{eq:s}
\end{equation}
Let $\iota$ be an integer chosen uniformly at random from $\big\{1,2,\ldots,s\big\}$.
Finally, define 
\begin{equation}
   \Gamma_\ml = \Gamma_1 \cup J_1 \cup J_2 \cup \cdots \cup J_\iota
   \quad\text{and}\quad
   \Gamma_\mr = \Gamma_3 \cup J_\iota \cup J_{\iota+1} \cup \cdots \cup J_s.
   \label{eq:gammalr}
\end{equation}
Then, $\Gamma_\ml$ represents the left portion of $\Gamma$, 
$\Gamma_\mr$ represents the right portion of $\Gamma$, and 
$\Gamma_\ml\cap\Gamma_\mr=J_\iota$.

We need to introduce some more notation to be precise about boundary conditions. For any $\si\in\tilde\O$, $A\subset \Lkn$, if $\si=\{\si_x,\;x\in\Lkn\}$ then we write $\si_A$ for the set of edges $\{\si_x\,,\;x\in A\}$.  
If $\xi=\si_A$ for some $\si\in\tilde\O$ and $A\subset\Lkn$ we say that $\si$ {\it contains\/} $\xi$ and we call $\xi$ a {\it partial triangulation\/} in $\tilde\O$. If $A\cap A'=\emptyset$ and $\xi=\si_A$, $\xi'=\si_{A'}$ for some $\si\in\tilde\O$, then we define $\xi\cup\xi'=\si_{A\cup A'}$. 

We use partial triangulations $\xi$ in $\tilde\O$ as boundary conditions for a region $B\subset\Gamma$. Fix a partial triangulation $\xi$. We denote by $A_\xi\subset\Lkn$ the set of midpoints of the edges in $\xi$. Let $\tilde\O^\xi$ denote the set of full triangulations $\si\in\tilde\O$ that contain $\xi$. We define 
%agree with $\xi$ where $\xi $ is defined.  that is as an assignment of edges to midpoints in 
%$\Lkn\setminus A$, such that each edge is smaller than $C\log n$. 
for any $B\subset \Gamma$, and any $\xi$ such that $A_\xi\subset \Lkn\setminus B$,
\begin{equation}
   \tilde\Omega_B^\xi
   = \{\si_B:\;\sigma\in\tilde\Omega^\xi\}.
   \label{eq:ensemble}
\end{equation}
For any $\eta_B\in \tilde\Omega_B^\xi$, let 
$$
   \mu_B^\xi(\eta_B)=\frac{\sum_{\si\in\tilde\O^\xi:\; \si_B=\eta_B}\tilde\mu(\sigma)}{\tilde\mu(\tilde\Omega^\xi)}\,,
 %  = \frac{\lambda^{\sum_{x\in A}|\sigma_x|}}{\sum_{\sigma'\in\tilde\Omega_\Gamma^\xi}\lambda^{\sum_{x\in\Gamma}|\sigma_x|}},
$$ 
be the induced probability measure over $\tilde\Omega_B^\xi$. In words, $ \mu_B^\xi$ is the marginal distribution over midpoints $B$ when we impose a boundary condition $\xi$. If $\xi$ is empty (no boundary condition) we simply write $\tilde\Omega_B$ and $\mu_B$.
%We remark that in this definition we do not assume that 
%$\xi$ contains all edges with midpoint in $\Lkn\setminus\Gamma$, that is in this case 
%$\tilde\Omega_\Gamma^\xi$ is a set of \emph{partial} triangulations of $\Lkn$ (each triangulation containing the edges of 
%$\xi$ and edges of midpoint in $\Gamma$), and $\mu_\Gamma^\xi$ is a probability measure on this set of partial triangulations.
%For any $A\subset \Gamma$ we also write $\tilde\Omega_A^\xi$ for the set of triangulations $\si\in\tilde\Omega_\Gamma^\xi$ such that 
\begin{lemma}\label{lem:decay}
   %Let $\sigma_\ml$ and $\sigma_\mr$ be any partial triangulations of $\Gamma_\ml\setminus J_\iota$ and $\Gamma_\mr\setminus J_\iota$, respectively.
   There exists a positive constant $c=c(\l,k)$ such that for any partial triangulation $\xi $ with $A_\xi\subset\Lkn\setminus\Gamma$,
   for all functions $f_\ml,f_\mr:\tilde\O\mapsto \bbR$ such that $f_\ml$ depends only on edges with midpoint in $\Gamma_\ml\setminus J_\iota$ and $f_\mr$ depends only on edges with midpoint in $\Gamma_\mr\setminus J_\iota$,  
   and for any $\sigma_\ml \in \tilde\Omega_{\Gamma_\ml\setminus J_\iota}^\xi$ and $\sigma_\mr\in\tilde\Omega_{\Gamma_\mr\setminus J_\iota}^\xi$,
   we have 
   $$
      \big|\mu_{\Gamma_\ml\setminus J_\iota}^{\xi\cup\sigma_\mr}(f_\ml)-\mu_{\Gamma_\ml\setminus J_\iota}^{\xi}(f_\ml)\big|
      \leq  \mu_{\Gamma_\ml\setminus J_\iota}^{\xi}(|f_\ml|) \exp(-c \log n)
   $$
   and
   $$
      \big|\mu_{\Gamma_\mr\setminus J_\iota}^{\xi\cup\sigma_\ml}(f_\mr)-\mu_{\Gamma_\mr\setminus J_\iota}^{\xi}(f_\mr)\big|
      \leq \mu_{\Gamma_\mr\setminus J_\iota}^{\xi}(|f_\mr|)   \exp(-c\log n).
   $$
\end{lemma}
\begin{proof}
   We will establish only the first estimate; the second follows by a symmetrical argument.
   Since $f_\ml$ depends only on edges with midpoint in $\Gamma_\ml\setminus J_\iota$, it is enough to show that, 
   for any $\sigma_\mr\in\tilde\Omega_{\Gamma_\mr\setminus J_\iota}^\xi$ 
   and any $\tau\in\tilde\Omega^{\xi\cup\sigma_\mr}_{\Gamma_\ml\setminus J_\iota}$, 
   we have 
   \begin{equation}\label{cice}
      \left|\frac{\mu_{\Gamma_\ml\setminus J_\iota}^{\xi\cup\sigma_\mr}(\tau)}{\mu_{\Gamma_\ml\setminus J_\iota}^\xi(\tau)}-1\right|\leq \exp(-c_2C\log n),
   \end{equation}
   for some positive $c_2=c_2(\l,k)$, where $C$ is the constant in the definition of the width of $J_\iota$. 
   
   Let $\eta$ and $\eta'$ be random triangulations distributed as $\mu_{\Gamma_\ml}^{\xi\cup\sigma_\mr}$ and 
   $\mu_{\Gamma_\ml}^{\xi}$, respectively. 
   Let $\mathbb{P}$ denote the following coupling between $\eta$ and $\eta'$; refer to Figure~\ref{fig:coupling}. 
%    First, sample $\eta'_{\Gamma_\ml\setminus J_\iota}=\t$ with the correct marginal distribution $\mu_{\Gamma_\ml\setminus J_\iota}^\xi$. Next, 
   The idea is to sample recursively edges from the pair $(\eta,\eta')$ in vertical strips inside $J_\iota$ from right to left
   from a suitable coupling of $\mu_{J_\iota}^{\xi}$ and $\mu_{J_\iota}^{\xi\cup\si_\mr}$. Here we will use the estimate of Lemma~\ref{lem:toptobottom} to ensure that, with large probability, there is a common top-to-bottom crossing of unit verticals within $J_\iota$. On this event we can safely resample $(\eta_{\Gamma_\ml\setminus J_\iota},\eta'_{\Gamma_\ml\setminus J_\iota})$ in such a way that $ \eta_{\Gamma_\ml\setminus J_\iota}=\eta'_{\Gamma_\ml\setminus J_\iota}=\t$. 
   
   We now present the details.
   Consider the midpoints of $\Gamma$ in order of their horizontal coordinate, from largest to smallest 
   (i.e., from right to left in Figure~\ref{fig:coupling}). 
   Let $v_0$ be the leftmost integer horizontal coordinate of points in $\Gamma_\mr\setminus J_\iota$, and let $V_0=\xi\cup \sigma_\mr$ and $V_0'=\xi$.
   Now for $i\geq1$, define $v_i,V_i,V_i'$ inductively as follows.
   Let $v_i< v_{i-1}$ be the rightmost integer horizontal coordinate 
   that is not crossed by an edge of $V_0 \cup V_1 \cup V_1' \cup V_2 \cup V_2' \cup \cdots \cup V_{i-1} \cup V_{i-1}'$. 
   Using the coupling from Lemma~\ref{lem:toptobottom}, sample all edges of $\eta$ and $\eta'$ whose midpoints have horizontal coordinate $v_{i}$, and denote them by $V_{i}$ and $V_{i}'$, respectively.  
   There are two cases.
   In the first case, at least one edge of $V_{i}$ or $V_{i}'$ is not a unit vertical (as happens with $i=1,2$ and $3$ in Figure~\ref{fig:coupling}).
   In this case, continue by defining $v_{i+1}$ as described above.
   If $v_{i+1}$ is a horizontal coordinate in $J_\iota$, sample $V_{i+1}$ and $V_{i+1}'$ as described above and iterate.
   Otherwise, if $v_{i+1}$ is not in $J_\iota$, stop this procedure and sample the remaining edges of $\eta$ and $\eta'$ independently.
   In the second case, all edges in $V_{i}$ and $V_{i}'$ are unit verticals (i.e., they create a top-to-bottom crossing of $\Gamma$, 
   as in Figure~\ref{fig:coupling} for $i=4$).
   Then stop the procedure above and sample the edges with horizontal coordinate smaller than $v_i$ identically in both $\eta$ and $\eta'$
   (as depicted by the gray edges in Figure~\ref{fig:coupling}), and then
   sample the remaining edges (that necessarily have midpoints in $\Gamma_\mr$) independently in $\eta$ and $\eta'$.
   Let $I_{\eta,\eta'}$ be the event that $\eta$ and $\eta'$ have a common top-to-bottom crossing of unit verticals with midpoint in~$J_\iota$.
   \begin{figure}[htbp]
      \begin{center}
         \includegraphics[scale=.9]{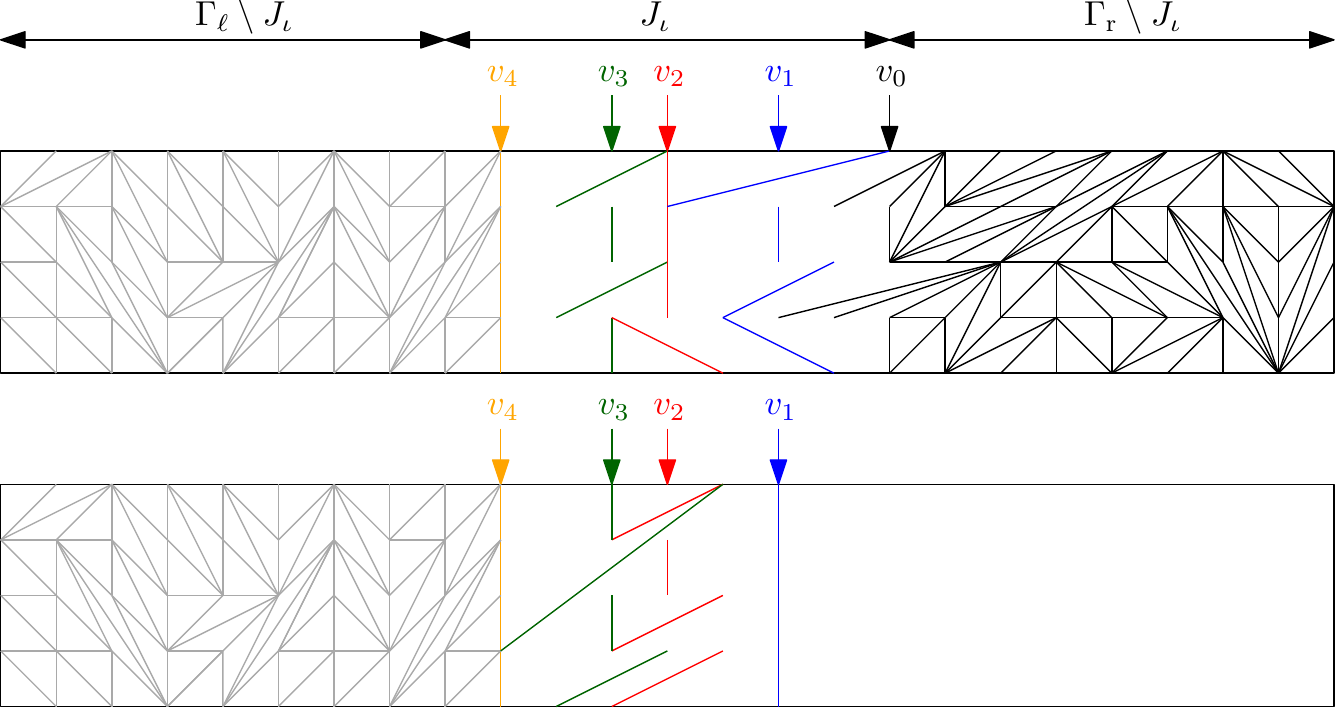}
      \end{center}\vspace{-.5cm}
      \caption{Coupling between $\mu_{\Gamma_\ml}^{\xi\cup\sigma_\mr}$ (above) and $\mu_{\Gamma_\ml}^{\xi}$ (below).
         Note that the figure is not to scale: in reality, the middle region $J_\iota$ is much smaller than the two outer regions.}
      \label{fig:coupling}
   \end{figure}

   Let $\eta_\ml,\eta'_\ml$ be the edges of $\eta,\eta'$ with midpoints in $\Gamma_\ml\setminus J_\iota$, 
   and let $\eta_\mr,\eta'_\mr$ be the edges of $\eta,\eta'$ with midpoints in $\Gamma_\mr\setminus J_\iota$.   
   Using the above coupling, 
   for any $\tau'\in\tilde\Omega^{\xi}_{\Gamma_\ml\setminus J_\iota}$ we obtain
   \begin{align*}
      \mu_{\Gamma_\ml\setminus J_\iota}^{\xi}(\tau')=\mathbb{P}(\eta'_\ml=\tau')
      &= \sum\nolimits_{\tau\in\tilde\Omega^{\xi\cup\sigma_\mr}_{\Gamma_\ml\setminus J_\iota}} \mathbb{P}(\eta_\ml = \tau, \eta_\ml'=\tau')\\
      &= \mathbb{P}(\eta_\ml = \tau', \eta_\ml'=\tau') + \sum\nolimits_{\tau\in\tilde\Omega^{\xi\cup\sigma_\mr}_{\Gamma_\ml\setminus J_\iota}\colon \tau \neq \tau'} \mathbb{P}(\eta_\ml = \tau, \eta_\ml'=\tau').
   \end{align*}
   The first term on the right-hand side above is at most $\mathbb{P}(\eta_\ml=\tau')=\mu_{\Gamma_\ml\setminus J_\iota}^{\xi\cup\sigma_\mr}(\tau')$.
   The second term is bounded above by 
   $$
      \mathbb{P}(\eta_\ml'=\tau')\mathbb{P}(\eta_\ml'\neq \eta_\ml \mid \eta_\ml'=\tau')
      \leq \mathbb{P}(\eta_\ml'=\tau')\mathbb{P}(I_{\eta,\eta'}^\compl\mid \eta_\ml'=\tau')
      \leq \mathbb{P}(\eta_\ml'=\tau')\exp(-4c C\log n),
   $$
   where the last step follows from Lemma~\ref{lem:toptobottom}.
   Plugging this into the equation above, and rearranging the terms, we obtain
   $$
      \mu_{\Gamma_\ml\setminus J_\iota}^{\xi\cup\sigma_\mr}(\tau')
      \geq \left(1-\exp(-4c C\log n)\right)\mu_{\Gamma_\ml\setminus J_\iota}^{\xi}(\tau'),
   $$
   which holds uniformly over $\tau'$ and $\sigma_\mr$. Similarly, we write 
   \begin{align*}
      \mu_{\Gamma_\ml\setminus J_\iota}^{\xi\cup\sigma_\mr}(\tau)=\mathbb{P}(\eta_\ml=\tau)
      &\leq\mathbb{P}(\eta_\ml'=\tau)+ \mathbb{P}(\eta_\ml=\tau)\mathbb{P}(I_{\eta,\eta'}^\compl\mid \eta_\ml=\tau)\\
      &\leq \mu_{\Gamma_\ml\setminus J_\iota}^{\xi}(\tau)+ \mu_{\Gamma_\ml\setminus J_\iota}^{\xi\cup\sigma_\mr}(\tau)\exp(-4cC\log n),
   \end{align*}
   and the proof of \eqref{cice} is completed by rearranging the terms and setting $c_2$ appropriately.
\end{proof}

%############################################################################################
\subsection{Recursion via bisection}\label{sec:recursion}
We consider slabs of different scales: we index the scale by $j$, where $j=0$ corresponds to the full slab $\Lkn$ of width $n$, while  at scale $j$, 
we have slabs of width $w$ roughly equal to $n2^{-j}$. 
The finest scale will be 
\begin{equation}
   j_* = \min\big\{j\geq 0 \colon n2^{-j}\leq C^6\log ^6n\big\};
   \label{eq:kappa}
\end{equation}
in particular,  $n2^{-j_*}\geq \tfrac12(C^6\log^6n)$.
Recall how slabs are split and the definition of $\iota$ from the construction of $\Gamma_\ml$ and $\Gamma_\mr$ in the paragraph culminating in~\eqref{eq:gammalr}. 

Consider a given scale $j\in\{0,\dots,j_*\}$, and let $\Gamma=\Gamma_j$ be a slab at scale $j$.
Set $W_0=n$, and define the intervals
$$ 
   W_j
   =\big[n2^{-j}-jC^3\log^3n,n2^{-j}+jC^3\log^3n\big]\,,\quad j=1,\dots,j_*
$$ 
Notice that our slab $\Gamma $ is obtained after $j$ steps of the bisection procedure, so
that  $\Gamma$ necessarily has width $w\in W_j$. 
%be the interval possible values for the 
%width of a slab at scale $j$. Fix a value $w\in W_j$ for the width of $\Gamma$. 
Let $\si\in\tilde \O$ be an arbitrary triangulation in the good ensemble 
and set $\xi=\si_{\Lkn\setminus \Gamma}\in\tilde\O_{\Lkn\setminus \Gamma}$ as a boundary condition for the region $\Gamma$. 
%: an assignment of edges to 
%$\Lkn\setminus \Gamma$ where each edge has length at most $C\log n$.
Consider 
the continuous time Markov chain on $\tilde\O^\xi_\Gamma$ with Dirichlet form
 \begin{equation}\label{dirgxi}
   \mathcal{E}_\Gamma^\xi(f,f)=\frac{1}{2}\sum_{\sigma_\G,\sigma'_\G\in\tilde\O^\xi_\Gamma}\mu_\G^\xi(\sigma_\G)\r_\G^\xi(\si_\G,\si'_\G)(f(\sigma_\G\cup \xi)-f(\sigma'_\G\cup\xi))^2,
\end{equation}
where $f:\tilde\O\mapsto\bbR$ and 
\begin{equation}\label{rates}
\r_\G^\xi(\si_\G,\si'_\G) = \frac{\l^{|\sigma'_\G\cup \xi|}}{\l^{|\sigma_\G\cup \xi|}+\l^{|\si'_\G\cup \xi|}}\ind{\sigma_\G\cup \xi\sim\sigma'_\G\cup \xi}\,.
\end{equation}
Let $\sobolev(\Gamma,\xi)$ denote the log-Sobolev constant defined as the smallest constant $c>0$ such that
\begin{equation}
   \ent_\Gamma^\xi(f^2) \leq c\,\mathcal{E}_\Gamma^\xi(f,f),
   \label{eq:sobolev2}
\end{equation}
holds for all functions $f$, where $ \ent_\Gamma^\xi(f^2)$ denotes the entropy of $f^2$ with respect to $\mu_\G^\xi$.

Finally we define, for each $j$, 
$$
   \gamma_j = \sup\big\{\sobolev(\Gamma,\xi) \colon \Gamma\subset \Lkn \text{ is a slab of width $w\in W_j$, and } \xi\in\tilde\Omega_{\Lkn\setminus \Gamma}\big\}.
$$

The following lemma summarizes the result of this recursion. 
\begin{lemma}\label{lem:recursion}
   There exists a positive constant $c_2$ such that, for any integer $j\in\{0,\dots,j_*-1\}$,
   $$
      \gamma_j
      \leq \left(1+e^{-c_2  \log n}\right)\left(1+\tfrac{4}{C^2\log^2n}\right)\gamma_{j+1}.
   $$
\end{lemma}
\begin{proof}
   Let $\Gamma$ be a fixed slab of width $w\in W_{j}$, and let $\xi$ be a given boundary condition.
   Let $s$, $\iota$, $\Gamma_\ml$ and $\Gamma_\mr$ be as described in the paragraph culminating in~\eqref{eq:gammalr}. 
   From Lemma~\ref{lem:decay} and \cite[Proposition 2.1]{Cesi}, for any function $f\colon \tilde\Omega\mapsto \mathbb{R}$ we have that $\ent_\Gamma^\xi(f^2)$ is bounded above by 
   %\NOTE{AS: Changed $1-e^{...}$ to $1+e^{...}$}
   \begin{equation}
\frac1{s}       \sum_{\iota=1}^s \left(1+e^{-c_2  \log n}\right)
         \left(\sum_{\sigma_\mr \in \tilde\Omega_{\Gamma_\mr\setminus J_\iota}^\xi}\mu_{\Gamma_\mr\setminus J_\iota}^\xi(\sigma_\mr)\ent_{\Gamma_\ml}^{\xi\cup\sigma_\mr}(f^2)
              +\sum_{\sigma_\ml \in \tilde\Omega_{\Gamma_\ml\setminus J_\iota}^\xi}\mu_{\Gamma_\ml\setminus J_\iota}^\xi(\sigma_\ml)\ent_{\Gamma_\mr}^{\xi\cup\sigma_\ml}(f^2)\right).
      \label{eq:recent1}
   \end{equation}
%   where $\ent_\Gamma^\xi(f^2)$ denotes the entropy of $f^2$ with respect to the probability measure $\mu_\Gamma^\xi$, 
%   and $c_2$ is some constant obtained from $\beta$ in Lemma~\ref{lem:decay}.
   Note that $\ent_{\Gamma_\ml}^{\xi\cup\sigma_\mr}(f^2)$ and $\ent_{\Gamma_\mr}^{\xi\cup\sigma_\ml}(f^2)$ are entropy functions for slabs on scale $j+1$ 
   given boundary conditions $\xi\cup\sigma_\mr$ and $\xi\cup\sigma_\ml$, respectively. 
   Therefore, by \eqref{eq:sobolev2} we have
    \begin{equation}\label{sumod1}
          \ent_{\Gamma_\ml}^{\xi\cup\sigma_\mr}(f^2)
      \leq \sobolev(\Gamma_\ml,\xi\cup\sigma_\mr)\mathcal{E}_{\Gamma_\ml}^{\xi\cup\sigma_\mr}(f,f)
      \leq \gamma_{j+1}\mathcal{E}_{\Gamma_\ml}^{\xi\cup\sigma_\mr}(f,f),
   \end{equation} and similarly for the second term in~\eqref{eq:recent1}. 
Now we claim that
    \begin{equation}\label{sumod}
   \sum_{\iota=1}^s 
         \left(\sum_{\sigma_\mr \in \tilde\Omega_{\Gamma_\mr\setminus J_\iota}^\xi}\mu_{\Gamma_\mr\setminus J_\iota}^\xi(\sigma_\mr)\mathcal{E}_{\Gamma_\ml}^{\xi\cup\sigma_\mr}(f,f)
              +\sum_{\sigma_\ml \in \tilde\Omega_{\Gamma_\ml\setminus J_\iota}^\xi}\mu_{\Gamma_\ml\setminus J_\iota}^\xi(\sigma_\ml)\mathcal{E}_{\Gamma_\mr}^{\xi\cup\sigma_\ml}(f,f)\right)\leq (1+s)\mathcal{E}_{\Gamma}^{\xi}(f,f).
\end{equation}
To prove \eqref{sumod} we proceed as follows. 
Since a given edge $\si_x$ in a triangulation has at most one value $\si'_x\neq \si_x$ it can flip to, we may write the flip rates \eqref{rates} as 
$$
\r_\G^\xi(\si_\G,\si'_\G) = \sum_{x\in\G}\frac{\l^{|\sigma'_x|}}{\l^{|\sigma_x|}+\l^{|\si'_x|}}
\ind{\sigma_\G\cup \xi\sim\sigma'_\G\cup \xi;\;\si_x\neq\si'_x}=:\sum_{x\in\G}\r_{x,\G}^\xi(\si_\G)\,.
$$
Therefore,
\begin{equation}\label{dirichxi}
\mathcal{E}_{\Gamma}^{\xi}(f,f)=\frac12 \sum_x\sum_{\si_\G\in\tilde\O_\G^\xi}\mu_\G^\xi(\si_\G)\r_{x,\G}^\xi(\si_\G)(\nabla_x f(\si_\G\cup\xi))^2,
\end{equation}
%\NOTE{AS: I added the (unimportant) $1/2$ in~\eqref{dirichxi} and the two display eqn below}
where we use $\nabla_x f$ to denote the difference in values of $f$ before and after the flip at $x$.
It follows that
\begin{align*}
&\sum_{\sigma_\mr \in \tilde\Omega_{\Gamma_\mr\setminus J_\iota}^\xi}\mu_{\Gamma_\mr\setminus J_\iota}^\xi(\sigma_\mr)\mathcal{E}_{\Gamma_\ml}^{\xi\cup\sigma_\mr}(f,f) \\
&\qquad = \frac{1}{2}\sum_{x\in\G_\ml}
\sum_{\sigma_\mr \in \tilde\Omega_{\Gamma_\mr\setminus J_\iota}^\xi}\mu_{\Gamma_\mr\setminus J_\iota}^\xi(\sigma_\mr)\sum_{\eta_{\G_\ml} \in \tilde\Omega_{\Gamma_\ml}^{\xi\cup\si_\mr}}\mu_{\Gamma_\ml}^{\xi\cup\si_\mr}(\eta_{\G_\ml})\r_{x,\G_\ml}^{\xi\cup\si_\mr}(\eta_{\G_\ml})(\nabla_x f(\si_{\eta_{\G_\ml}}\cup\xi\cup\si_\mr))^2,
\end{align*}
where, as before, we use the shortcut notation $\si_\mr=\si_{\G_\mr\setminus J_\iota}$.
Using
$$
\mu_{\Gamma_\mr\setminus J_\iota}^\xi(\sigma_\mr)\mu_{\Gamma_\ml}^{\xi\cup\si_\mr}(\eta_{\G_\ml})\r_{x,\G_\ml}^{\xi\cup\si_\mr}(\eta_{\G_\ml}) = \mu_\G^\xi(\eta_{\G_\ml}\cup\si_\mr)\r_{x,\G}^{\xi}(\eta_{\G_\ml}\cup\si_\mr)
$$
and rearranging the sum, we obtain
\begin{align*}
&\sum_{\sigma_\mr \in \tilde\Omega_{\Gamma_\mr\setminus J_\iota}^\xi}\mu_{\Gamma_\mr\setminus J_\iota}^\xi(\sigma_\mr)\mathcal{E}_{\Gamma_\ml}^{\xi\cup\sigma_\mr}(f,f) 
= \frac{1}{2}\sum_{x\in\G_\ml}
\sum_{\sigma_\G \in \tilde\Omega_{\Gamma}^\xi}\mu_{\Gamma}^{\xi}(\si_\G)\r_{x,\G}^{\xi}(\si_\G)(\nabla_x f(\si_{\G}\cup\xi))^2.
\end{align*}
A similar expression holds for the second term on the left-hand side of \eqref{sumod}, and the desired estimate follows from the expression \eqref{dirichxi}.
 
   Plugging \eqref{sumod} and \eqref{sumod1} into the bound in~\eqref{eq:recent1} we have 
       \begin{align*}
   \ent_\Gamma^\xi(f^2)   \leq \left(1+e^{-c_2  \log n}\right)\gamma_{j+1}\left(1+\tfrac1s\right) \mathcal{E}_{\Gamma}^{\xi}(f,f).
   \end{align*}
   This establishes that $\sobolev(\Gamma,\xi)\leq \left(1+e^{-c_2  \log n}\right)\gamma_{j+1}\left(1+\tfrac{1}{s}\right)$. 
   Since this bound does not 
   depend on $\xi$ and the choice of slab $\Gamma$ at scale $j$, the proof is completed by using the value of $s$ from~\eqref{eq:s}.
\end{proof}

We conclude the proof with the base of the induction.
\begin{lemma}\label{lem:base}
   There exists a constant $c=c(\lambda,k)$ such that 
   $$
      \gamma_{j_*} \leq \log^{c} n.
   $$
\end{lemma}
\begin{proof}
   Let $\Gamma$ be a slab at scale $j_*$, so that  the width of $\Gamma$ is of order $\log^6n$.
   Let $\xi\in\tilde\Omega_{\Lkn\setminus\Gamma}$ be a boundary condition.
   We note that the argument of Theorem~\ref{thm:polybound} can be repeated with no modifications for the chain restricted to the good set $\tilde \O$. Therefore, there exists a constant $c_1=c_1(\lambda,k)$ independent of $\Gamma$ and $\xi$ such that 
   the relaxation time of the discrete time chain on $\G$ with boundary condition $\xi$ is at most $\log^{c_1}n$.  Passing to continuous time, we have that $\tilde T_\rel(\Gamma,\xi)\leq \log^{c_1}n$.  
   Since triangulations in $\tilde\Omega_\Gamma^\xi$ have edges of length at most $C\log n$, 
   there exists a constant $c_2$ 
   such that 
   $$
      \min_{\sigma_\G\in \tilde\Omega_\Gamma^\xi}\mu_\Gamma^\xi(\sigma_\G) \geq n^{-c_2},
   $$
   uniformly over all slabs $\Gamma$ at scale $j_*$ and boundary conditions $\xi$.  
   Therefore, using the relation between the relaxation time and the log-Sobolev constant
   from~\eqref{eq:sobolev} we have that 
   $$
      \sobolev(\Gamma,\xi) \leq 
      \tilde T_\rel(\Gamma,\xi)\, \left(\frac{\log(n^{c_2})}{1/2}\right).
   $$
   Since the bound above is uniform~$\Gamma$ and~$\xi$, this proves the desired bound on $\gamma_{j_*}$.
\end{proof}

%############################################################################################
\subsection{Completing the proof}

% We established the following theorem.
% \begin{theorem}
%    For any $\lambda<1$, there exists a constant $c=c(\lambda,k)>0$ so that the relaxation time of 
%    the continuous time edge-flipping Markov chain on $\tilde\Omega$ is at most $\log^cn$.
% \end{theorem}
\begin{proof}[Proof of Theorem~\ref{mainth}]
   We start by bounding the mixing time of the discrete time Markov chain on $\tilde\Omega$.
   Lemma~\ref{lem:recursion} implies that the log-Sobolev constant of the continuous time Markov chain on $\Lkn$ with no boundary condition is at most 
   $$
      \sobolev(\Lkn) \leq \gamma_0
      \leq \left(1+e^{-c_2 \log n}\right)^{j_*-1}\left(1+\tfrac{4}{C^2\log^2n}\right)^{j_*-1}\gamma_{j_*}
      \leq 2\gamma_{j_*},
   $$
   where the last step follows since $j_*\leq \log_2 n$.
   Also, we have that 
   $$
      \min_{\sigma\in \tilde\Omega}\mu(\sigma)
      \geq \frac{\lambda^{|\Lkn|C\log n}}{(2\lambda)^{|\Lkn|}},
   $$
   where $(2\lambda)^{|\Lkn|}$ comes from Anclin's bound of $2^{|\Lkn|}$ for the number of lattice triangulations \cite{Anclin:2003jb}, and the fact that the 
   total edge length of any triangulation is at least $|\Lkn|$.
   Therefore, using the relation between the mixing time and log-Sobolev constant in~\eqref{eq:sobolev}, we deduce that the mixing time $\tilde T_\mix$ of the continuous time Markov chain on $\tilde\O$ is bounded above by  $c\gamma_{j_*}\log n$. Thus, 
   the mixing time of the discrete chain in $\tilde\Omega$ is at most 
   $|\Lkn| c\gamma_{j_*}\log n$,
   for some constant $c$. Using Lemma~\ref{lem:base} and the fact that $|\Lkn|$ is of order $nk$, 
   we obtain that the mixing time of the Markov chain restricted to 
   $\tilde \Omega$ is at most $c n \log^{c}n$, for some new positive constant $c$ (which depends
   on~$k$ and~$\lambda$).
   
   Now we compare the restricted chain on $\tilde\O$ to the original unrestricted chain on $\O=\Okn$.  
   Let $T_1=c n \log^{c}n$ and fix the constant $c>0$  so that the total variation distance between the restricted chain at time $T_1$ and the restricted stationary distribution $\tilde \mu$ is 
   at most $1/8$.
   We obtain the mixing time of the  unrestricted chain 
   via the following coupling.
   % with the Markov chain restricted to the good ensemble $\tilde \Omega$. 
   Let $T_0=c_1n^2$, where $c_1$ is the constant in Lemma~\ref{lem:burnin}. 
   Let the unrestricted Markov chain run for $T_0+T_1$ steps. 
   With probability at least $1-n^{-2}$, the unrestricted chain never leaves the set $\tilde\Omega$ during the time interval $[T_0,T_0+T_1]$; therefore, we 
   can couple its steps with those of the restricted chain. This gives that the total variation distance between the unrestricted chain at time $T_0+T_1$ 
   and the stationary distribution is at most $n^{-2} + 1/8 + \mu(\Omega\setminus \tilde\Omega)$.
   Since $\Omega\setminus \tilde\Omega$ only contains triangulations for which the largest edge is larger than $C\log n$, Lemma~\ref{lem:tailedge} ensures that
$\mu(\Omega\setminus \tilde\Omega) \leq n^{-2}$ for large enough~$C$, and therefore 
   the total variation distance between the unrestricted chain at time $T_0+T_1$ and its stationary distribution is at most $1/4$.  This completes the proof of Theorem~\ref{mainth}.
\end{proof}

%############################################################################################
%############################################################################################
%############################################################################################
\bibliographystyle{plain}
\bibliography{triang_slabs}

\end{document}